\documentclass[11pt, oneside]{article}   	
\usepackage{geometry}               		
\usepackage[utf8]{inputenc}
\usepackage{amssymb, amsmath, amsfonts, amstext}

\title{The diameter of the fractional matching polytope and its hardness implications}
\author{Laura Sanit\`a \\
Combinatorics $\&$ Optimization Department, \\ University of Waterloo \\
\texttt{lsanita@uwaterloo.ca}
}

\date{}

\usepackage[numbers]{natbib}
\usepackage{graphicx}
\usepackage{amsthm}
\usepackage{float}

\newtheorem{theorem}{Theorem}

\newtheorem{lemma}{Lemma}

\newtheorem{obs}{Observation}
\newtheorem{definition}{Definition}

\begin{document}

\maketitle

\begin{abstract}
The (combinatorial) diameter of a polytope $P \subseteq \mathbb R^d$ is the maximum value of a shortest path between a pair of vertices 
on the 1-skeleton of $P$, that is the graph where the nodes are given by the $0$-dimensional faces of $P$, and the edges are given the 1-dimensional faces of $P$.   
The diameter of a polytope has been studied from many different perspectives, including a computational complexity point of view. In particular, [Frieze and Teng, 1994] showed that computing the diameter of a polytope is (weakly) NP-hard.

In this paper, we show that the problem of computing the diameter is strongly NP-hard even for a polytope with a very simple structure: namely, the \emph{fractional matching} polytope. We also show that computing a pair of vertices
at maximum shortest path distance on the 1-skeleton of this polytope is an APX-hard problem. We prove these results by giving an \emph{exact characterization}
of the diameter of the fractional matching polytope, that is of independent interest.

\end{abstract}

\section{Introduction}
The (combinatorial) diameter of a polytope $P \subseteq \mathbb R^d$ is the maximum value of a shortest path between a pair of vertices on the 1-skeleton of $P$, which is the graph where the vertices correspond to the 0-dimensional faces of $P$, and the edges are given by the 1-dimensional faces of $P$.
Giving bounds on the diameter of a polytope is a central question in discrete mathematics and computational geometry.
Despite decades of studies, it is still not known
 whether the diameter of a $d$-dimensional 
polytope with $n$ facets can be bounded by a polynomial function of $n$ and $d$ -- this is currently referred to as the 
\emph{polynomial Hirsch Conjecture} \cite{S12}.  
Besides being a fundamental open question in polyhedral theory, the importance of the
conjecture is also due to the fact that any polynomial pivot rule for the Simplex algorithm for Linear Programming 
requires the conjecture to hold. The existence of such rule would have significant consequences
on  
the existence of 
a strongly-polynomial time algorithm for Linear Programming. The latter
is a main open problem, mentioned in the list of the 
“mathematical problems for the next century”
given by S. Smale \cite{S00}. 

The study of diameters of polytopes has a rich history, and we refer e.g. to \cite{S13} for a survey.  
The polynomial Hirsch conjecture is a generalization of a conjecture proposed by Hirsch in
1957, which states that the combinatorial diameter of any $d$-dimensional polytope with
$n$ facets is at most $n-d$. This conjecture has been disproved first for unbounded
polyhedra \cite{KW67} and then for bounded polytopes \cite{S12}, though it is known to hold
for many classes of polytopes, such as for $0/1$ polytopes \cite{N89}. 
For the currently best known upper bound on the diameter we refer to 
\cite{N17}. Besides providing general bounds, many researchers in the past 50 years have given bounds and/or characterizations of the diameter of 
 polytopes that correspond to the set of feasible solutions of classical combinatorial optimization problems.
 Just to mention a few, such problems include matching \cite{BR74,C75}, TSP \cite{PR74,RC98},
 edge cover \cite{H91}, fractional stable set \cite{MS14}, network flow and transportation problems \cite{Ba84,BFH17,BDF17,BHS06},
 stable marriage \cite{EMM14}, and many more.

\smallskip
The diameter of a polytope has been studied from many different perspectives,
including a
computational complexity point of view. 
In particular, Frieze and Teng \cite{FT94} showed that
computing the diameter of a polytope is (weakly) NP-hard. 
Digging more into the complexity of computing the diameter of a polytope remains  
an interesting problem (see e.g. problem 10 on the list of
35 algorithmic problems in polytope theory, given in \cite{KP03}).

In this paper, we show that
computing the diameter of a polytope is a strongly NP-hard problem, and finding a pair of vertices at maximum (shortest path) distance on the 1-skeleton of a polytope is an APX-hard problem. 
In fact, what is probably more interesting, is that we can show hardness
already for a polytope that has a very simple structure and it is quite well-understood: namely,
the \emph{fractional matching} polytope.
We achieve these results by giving an exact characterization of the diameter of such polytope, 
which technically constitutes the main result of this paper, and is of independent interest.
We are going to describe such characterization next.

One well-studied polytope for which a characterization of the diameter is known since the mid 70's, 
is the \emph{matching} polytope, that is the polytope given by the convex combination
of the characteristic vectors of matchings of a graph. 
Formally, given a simple graph $G=(V,E)$ with $n$ nodes and $m$ edges, the matching polytope $\mathcal P_M$ is as follows \cite{E65}:

\begin{equation}
\label{eq:match_polytope}
\begin{array}{lll}
\mathcal P_M:= \{x \in \mathbb R^m:  &  x(\delta(v)) \leq 1 &  \forall v \in V, \\
 & x(E(S)) \leq \frac{|S|-1}{2} & \forall S\subseteq V \mbox{ with } |S| \geq 3, |S| \mbox{ odd},\\
 & x\geq 0 \}.&
\end{array}
\end{equation}

\noindent
Here $\delta(v)$ denotes the set of all edges of $G$ incident into the node $v$, $E(S)$ denotes the set of edges with both endpoints in $S$,
and for a set $F \subseteq E$, $x(F) =\sum_{e \in F} x_e$.
The matching polytope is one of the most studied polytopes in combinatorial optimization. 
As shown in \cite{BR74,C75} the diameter of the matching polytope is equal to the maximum cardinality of a matching in $G$, i.e. if we denote by $diam(\mathcal P)$ and by $vert(\mathcal P)$ the diameter and the set of vertices of a polytope $\mathcal P$, respectively, 
the result in \cite{BR74,C75} states 

$$diam(\mathcal P_M)= \max_{x \in vert(\mathcal P_M)}\{{\bf 1}^T x\}.$$

\medskip
The \emph{fractional} matching polytope $\mathcal P_{FM}$ is the polytope associated with the LP-relaxation of the standard IP formulation for the matching problem, and it is given by dropping the so-called odd-set inequalities from (\ref{eq:match_polytope}): 
\begin{equation}
\label{eq:polytope}
\mathcal P_{FM}:= \{x \in \mathbb R^m: \;\; x(\delta(v)) \leq 1 \;\; \forall v \in V, \;\; x\geq 0\}.
\end{equation}

As for the matching polytope, this polytope has been extensively studied in the optimization community. It is well known
to be an half-integral polytope, and many structural results about its vertices, 
faces, and adjacency of the vertices are known. In particular, the support of a vertex $x$ of $\mathcal P_{FM}$
is the union of a matching, denoted by $\mathcal M_x$,  given by the edges $\{e\in E: x_e=1\}$, and a collection of node-disjoint odd cycles, denoted by $\mathcal C_x$, given by the edges $\{e\in E: x_e=\frac{1}{2}\}$, as shown in \cite{Ba70}.

Our main result is a characterization of the diameter of this polytope, given in Theorem \ref{thm:main_result}.

\begin{theorem}
\label{thm:main_result} The diameter of the fractional matching polytope $\mathcal P_{FM}$ is 
$$diam(\mathcal P_{FM})= \max_{x \in vert(\mathcal P_{FM})}\{{\bf 1}^T x + \frac{|\mathcal C_x|}{2}\}.$$
\end{theorem}

As an observation, note that if the graph $G$ is bipartite, then $\mathcal C_x =\emptyset$ for all vertices $x$ of $
\mathcal P_{FM}$, and our diameter bound for $\mathcal P_{FM}$ reduces to the one for $\mathcal P_{M}$, 
as expected, since the odd-set inequalities are redundant for bipartite graphs.

\smallskip
As already mentioned, our result has an important algorithmic consequence regarding the hardness of computing the diameter of a polytope. 
As we will show later, combining Theorem \ref{thm:main_result} with an easy reduction
from the NP-complete problem of finding a set of triangles partitioning the vertices of a graph, 
one can easily get that computing the diameter of a polytope is a strongly NP-hard problem.

\begin{theorem}
\label{thm:second_result} Computing the diameter of a polytope is a strongly NP-hard problem.
\end{theorem}

With some extra effort, the hardness result can be strengthen to get APX-hardness for the optimization problem of finding a pair of vertices at maximum distance on the 1-skeleton of a polytope.

\begin{theorem}
\label{thm:third_result} Finding a pair of vertices
at maximum (shortest path) distance on the 1-skeleton of a polytope is an APX-hard problem.
\end{theorem}

We conclude this introduction with a brief description of the strategy we follow to give our characterization.
We prove our result in two steps: first, we show that the value given
in Theorem \ref{thm:main_result} is an \emph{upper} bound on the diameter, and then 
we show that it is a \emph{lower} bound on the distance between two specific vertices
of $\mathcal P_{FM}$. Proving the first step in particular requires some effort. 

The upper bound proof is based on a \emph{token} argument. 
Specifically, given two distinct vertices $z$ and $y$ of $\mathcal P_{FM}$, we prove that the distance between $z$ and $y$ is bounded by 
${\bf 1}^Tw + \frac{|\mathcal C_w|}{2}$ for some vertex $w$ of $\mathcal P_{FM}$ 
whose support graph is in the union of the support graphs of $z$ and $y$. 
We will assign to each node $v \in V$ with $w(\delta(v))=1$ a token of value $\frac{1}{2}$, and to each cycle $C \in \mathcal C_w$ a token of value $\frac{1}{2}$ (note that the total sum of the token values equals ${\bf 1}^T w + \frac{|\mathcal C_w|}{2}$). We will then define a path on the 1-skeleton of $\mathcal P_{FM}$ from $z$ to $y$, and show that each \emph{move} along this path can be payed by using two tokens from some nodes or cycles,
where a move refers to the process of traversing an edge on the 1-skeleton of $\mathcal P_{FM}$.

 We would like to highlight that the selection of the moves to take is not straightforward.
 A trivial attempt to define a $z$-$y$ path could be to (i) define a path from $z$ to a $0/1$-vertex $\bar z$ obtained by ``rounding'' the half-valued coordinates of one of the cycles in $\mathcal C_z$ at each step, (ii) define a path from $y$ to a $0/1$-vertex $\bar y$ obtained by ``rounding'' the half-valued coordinates of one of the cycles in $\mathcal C_y$ at each step, and (iii) define a path from $\bar z$ to $\bar y$ 
 guided by the symmetric difference between the matchings corresponding to $\bar z$ and $\bar y$ (indeed, 
 this corresponds to a $\bar z$-$\bar y$ path on the 1-skeleton of the matching polytope $\mathcal P_{M}$). 
 Unfortunately, it is not difficult to construct examples where
 \emph{any} path of this form has a length strictly larger than the claimed bound (see Example A in the appendix). 
In order to reduce the number of moves, we will have to exploit better the adjacency properties of the vertices,
and we will have to keep track of the tokens we use to pay for our moves in a very careful way.

\section{Preliminaries}
In this section, we will introduce some notations and state some known structural results regarding the vertices
of the polytope $\mathcal P_{FM}$. 

To avoid confusion, we will always refer to the extreme points of a given polytope as \emph{vertices}, and to  the elements of $V$ of a given graph $G=(V,E)$ as \emph{nodes}. Furthermore, for a generic (sub)graph $H$, we will denote by $V(H)$ and $E(H)$ its set of nodes and edges, respectively. 
We say that a cycle $C$ (resp. a path $P$) is odd if $|E(C)|$ (resp. $|E(P)|$) is odd.
Given a matching $M$, an $M$-alternating path is a path that alternates edges in $M$ and edges not in $M$.
A node is $M$-exposed if it is not the endpoint of an edge in $M$.
An $M$-augmenting path is an $M$-alternating path whose endpoints are $M$-exposed.

We start by stating the well-known half-integrality property of the vertices of $\mathcal P_{FM}$. 
\begin{theorem}[\cite{Ba70}]\label{thm:balinski}
Every basic feasible solution to $\mathcal P_{FM}$ has components equal to
$0 , 1$ or $\frac{1}{2}$, and the edges with half-integral components
induce node-disjoint odd cycles. 
Furthermore, every half-integral vector $x \in \mathcal P_{FM}$
such that the half-integral components of $x$
induce node-disjoint odd cycles, is a vertex of $\mathcal P_{FM}$.
\end{theorem}

For a given vector $x \in \mathbb R^E$, we will refer to $\mathcal G_x$ as the graph induced by the support of $x$,
i.e. the graph induced by the set of edges $\{e\in E: x_e >0\}$. 
As already mentioned in the introduction, if $x$ is a vertex of $\mathcal P_{FM}$, then $\mathcal G_x$ is the union of a matching (denoted by $\mathcal M_x$) induced by the edges $\{e\in E: x_e=1\}$ and a collection of node-disjoint odd cycles (denoted by $\mathcal C_x$) induced by the edges $\{e\in E: x_e=\frac{1}{2}\}$. Furthermore, given two vectors $x,y \in \mathbb R^E$, 
 we let $\mathcal G_x \Delta \mathcal G_y$
 be the graph induced by the set of edges $\{e\in E: x_e \neq y_e\}$.

The following definition will be highly used later.

\begin{definition}
We say that an odd cycle $C$ of $G$ is \emph{packed} by
a matching $M$ (resp. by a vertex $x$), if $|M \cap E(C)| = \frac{|C|-1}{2}$ (resp. if $|\mathcal M_x \cap E(C)| = \frac{|C|-1}{2}$).   
\end{definition}

The paper in \cite{B13} gives an adjacency characterization for the vertices of the fractional \emph{perfect} $b$-matching polytope
 associated to a (possibly non-simple) graph. Using this result, it is 
 easy to derive adjacency properties for the vertices of $\mathcal P_{FM}$.
We here explicitly state a lemma that follows from Theorem 25 in \cite{B13}. This lemma gives some sufficient conditions for 
two vertices of $\mathcal P_{FM}$ to be adjacent. These are all the conditions that we will use in Section \ref{sec:diameter} to prove our upper bound
(refer to Figure~\ref{fig:adjacency} for some examples). 
To make the paper self-contained, we prove this lemma in the appendix.

\begin{lemma} (follows from  \cite{B13})
\label{cor:adjacency}
Let $y$ and $z$ be two vertices of $\mathcal P_{FM}$, and let $\mathcal C_z(y) \subseteq \mathcal C_z$ (resp. $\mathcal C_y(z) \subseteq \mathcal C_y$) be the set of cycles of $\mathcal G_z$ (resp. $\mathcal G_y$) packed
by $\mathcal M_y$ (resp. packed by $\mathcal M_z$).
 Then $y$ and $z$ are adjacent if $\mathcal G_{y} \Delta \mathcal G_z$ contains exactly one component $C$ 
 with this component being

\begin{itemize}
\item[(a)] one even cycle in $\mathcal M_z \Delta \mathcal M_y$, or
\item[(b)] one path in $\mathcal M_z \Delta \mathcal M_y$, or
\item[(c)] one odd cycle in $\mathcal C_z(y) \cup \mathcal C_y(z)$, or
\item[(d)] one odd cycle in $\mathcal C_z(y)$ and one odd cycle in $\mathcal C_y(z) $ intersecting in exactly one node, or
\item[(e)] two node-disjoint odd cycles in $\mathcal C_z(y) \cup \mathcal C_y(z)$ and a path $P \subseteq \mathcal M_y \Delta \mathcal M_z$ intersecting the cycles at its endpoints, or
\item[(f)] one odd cycle in $\mathcal C_z(y) \cup \mathcal C_y(z)$ and a path $P \subseteq \mathcal M_y \Delta \mathcal M_z$ intersecting the cycle at one of its endpoints.
\end{itemize}
\end{lemma}

\begin{figure}
\begin{center} 
\includegraphics[scale=0.45]{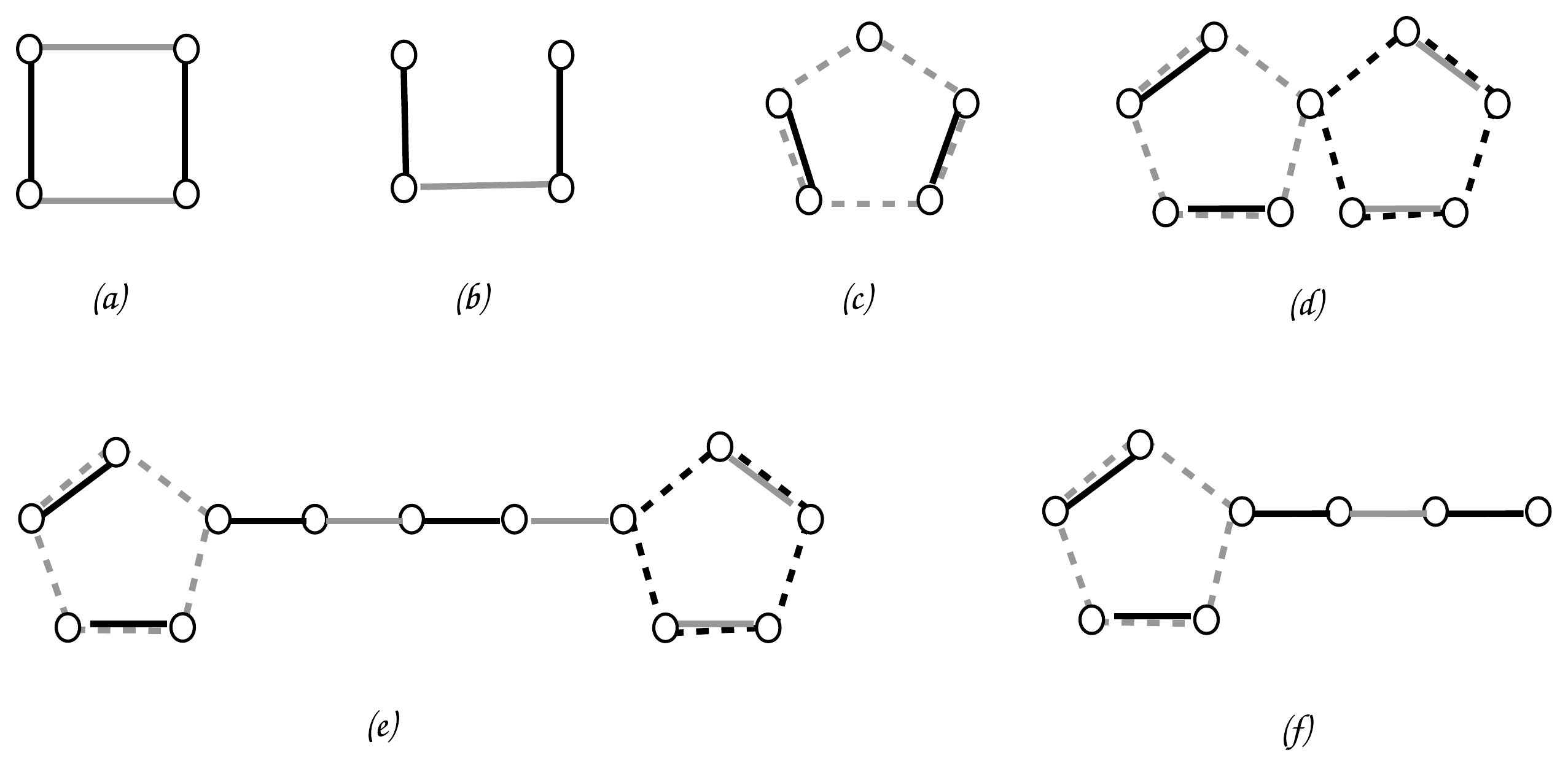}
\end{center}
\caption{Examples illustrating Lemma~\ref{cor:adjacency}. Black edges represent the edges in the support of $y$, and gray edges represent the edges in the support of $z$. Continuous lines represent edges of value 1, while dashed lines represent edges of value $\frac{1}{2}$.  
 }
\label{fig:adjacency} 
\end{figure}

Recall that we refer to the process of traversing an edge of a polytope (i.e. moving from one vertex to an adjacent one) as a \emph{move}.


\section{A characterization of the diameter of $\mathcal P_{FM}$}
\label{sec:diameter}
In this section we give a characterization of the diameter of the fractional matching polytope by providing a proof of Theorem \ref{thm:main_result}. 
Recall that Theorem \ref{thm:main_result} states
$$diam(\mathcal P_{FM}) = \max_{x \in vert(\mathcal P_{FM})} \{ {\bf 1}^Tx + \frac{|\mathcal C_x|}{2}\}.$$

We will prove that the right-hand side of the above formula is an upper bound on the diameter of $\mathcal P_{FM}$ 
in Sections 3.1-3.6, and
then prove it is a lower bound in Section \ref{sec:lower_bound}.

\subsection{Upper bound: general strategy} 
\label{sec:upper_bound_1}
Let us recall the strategy sketched in the introduction, regarding how we are going to prove our upper bound.
Given two distinct vertices $z$ and $y$ of $\mathcal P_{FM}$, we will prove that the distance between $z$ and $y$ is bounded by 
${\bf 1}^Tw + \frac{|\mathcal C_w|}{2}$ for some vertex $w$ of $\mathcal P_{FM}$ that will be identified later.

We assign to each node in $v $ in $V(\mathcal G_w)$ (i.e. with $w(\delta(v))=1$) a token of value $\frac{1}{2}$, and to each cycle $C \in \mathcal C_w$ a token of value $\frac{1}{2}$, so that the total sum of the token values equals ${\bf 1}^T w + \frac{|\mathcal C_w|}{2}$. 
We define a path on the 1-skeleton of $\mathcal P_{FM}$ from $z$ to $y$, and show that each move along this path can be payed by using two tokens from some nodes or cycles.
Specifically, our path will have the form 
$$z \rightarrow w \rightarrow r \rightarrow y$$ 
where $r$ and $w$ are (non necessarily distinct) vertices of $\mathcal P_{FM}$ with some nice structure, and the arrow ``$\rightarrow$'' indicates a path between the corresponding vertices on the 1-skeleton of $\mathcal P_{FM}$.
While traversing our path, we will satisfy some invariants, which will be helpful for keeping track of the used tokens. 
The first set of invariants, valid for the path from $z$ to $w$, are described in the next paragraph.

\paragraph{Invariants and definitions.} For every pair of consecutive vertices $\ell, \bar \ell$ on the path $z  \rightarrow w$ (with $\bar \ell$ coming after $ \ell$ when traversing this path), the following invariant will hold: 

\begin{equation}\label{eq:invariant}
 V(\mathcal G_{\ell}) \subseteq V(\mathcal G_{\bar \ell}).
 \end{equation}

In other words, $\bar \ell$ is obtained by ``augmenting'' $\ell$. 
This implies that each node $u$ with $\ell(\delta(u))=1$ will satisfy $w(\delta(u))=1$, and therefore
$u$ has been assigned a token. It follows that we can use the tokens of the nodes of $V(\mathcal G_{\bar \ell})$ to pay
for the move from $\ell$ to $\bar \ell$. 

We also anticipate here that for every pair of consecutive vertices $\ell, \bar \ell$ on the path $z  \rightarrow w$ (with $\bar \ell$ coming after $ \ell$ when traversing this path), the following invariant will hold: 

\begin{equation}\label{eq:invariant_cycle}
 (\mathcal C_{\ell} \setminus \mathcal C_{z}) \subseteq \mathcal C_{\bar \ell}.
 \end{equation}

This implies that \emph{if} an odd cycle $C$ appears for the first time in the support of a vertex $\ell \neq z$, then $C$ will be part
of the support of $w$ as well.  This ensures that $C$ has been assigned a token, and therefore we can use its token to pay
for the move we performed to arrive to $\ell$.

In the following, we will denote by $T(\ell)$ the subset of nodes of $V(\mathcal G_{\ell})$ that still have an available token, i.e., a token that has \emph{not} been used to pay for any of the moves performed to arrive from $z$ to $\ell$. 
Furthermore, for a given vertex $\ell$, we let $\tilde M_{\ell} \subseteq \mathcal M_{\ell}$ be the subset of edges of $\mathcal M_{\ell}$ with both endpoints not in $T(\ell)$, i.e. satisfying: $\{u,v\} \in \tilde M_{\ell} \Leftrightarrow \{u,v\} \in \mathcal M_{\ell} \mbox{ and } u,v \notin T(\ell)$. 

On our path from $z$ to $w$, we will satisfy two additional invariants.
The first one states that the set of edges whose endpoints have no tokens 
are a subset of the edges of $\mathcal G_y$:
\begin{equation}\label{eq:invariant2}
 \tilde M_{\ell} \subseteq E(\mathcal G_y)
 \end{equation}

The second one states that the nodes that used their tokens either can be ``paired up'' using edges of $\tilde M_{\ell}$,
or they  belong to cycles in $\mathcal C_{\ell} \cap \mathcal C_y$:

\begin{equation}\label{eq:invariant3}
 \mbox{ for each } v \in V(\mathcal G_{\ell}) \setminus V(\mathcal{C}_{\ell} \cap \mathcal{C}_y),  \mbox{ we have } v \in T(\ell) \Leftrightarrow \mbox{ $v$ is not an endpoint of an edge in }  \tilde M_{\ell}
 \end{equation}

\subsection{Moving from $z$ to $w$}
The vertex $w$ is obtained by ``augmenting'' $z$ (if possible) using edges of $\mathcal G_y$. 
This is done in three main steps. To describe the first one, we need to introduce a definition (refer to Figure~\ref{fig:critical}(a)).

\begin{definition}
\label{def:critical}
Let $\ell$ be a vertex, and $C$ be a cycle in $\mathcal C_y \setminus \mathcal C_{\ell}$,
such that $|V(C) \cap T(\ell)| \geq 2$. We say that the cycle  $C$ is \emph{critical} for $\ell$
if every odd path $P \subseteq C$ with nodes $\{v_1, \dots, v_k\}$ such that $V(P) \cap T(\ell) =\{v_1,v_k\}$, 
satisfies 
\begin{itemize}
\item[(i)] $|E(P)| \geq 3$, and
\item[(ii)] either $\{v_1,v_k\} \in \mathcal M_{\ell}$, or  $v_1,v_k \in V(\bar C)$ for some $\bar C \in \mathcal C_{\ell}$.
\end{itemize}
\end{definition}

Algorithm 1 describes the moves we perform to arrive to the vertex $w$, starting from $z$, and proceeds as follows. 
In Step 2 we consider the current vertex $\ell$ (with $\ell :=z$ at the beginning),
and we look for a critical cycle $C \in \mathcal C_{y}\setminus \mathcal C_{\ell}$. Given such a cycle, we perform a move which increases the number of nodes covered by $\mathcal M_{\ell}$ (see Figure~\ref{fig:critical}(b)). 
In Step 3 we look for $\tilde M_{\ell}$-augmenting paths whose edges are in $E(\mathcal G_y)$
and whose endpoints are  not in $V(\mathcal G_{\ell})$. If such a path is identified, we perform a move which again increases the number of nodes covered by $\mathcal M_{\ell}$.
In Step 4 we look for a cycle $C \in \mathcal C_y$, which is packed by $\ell$, and has a node $v \in V(C)$
which is not in $V(\mathcal G_{\ell})$. If such a cycle is identified, we perform a move which increases the number of cycles in $\mathcal C_{\ell} \cap \mathcal C_y$.

\begin{figure}[H]
\begin{center} 
\includegraphics[angle=0, width=8cm, height=5cm]{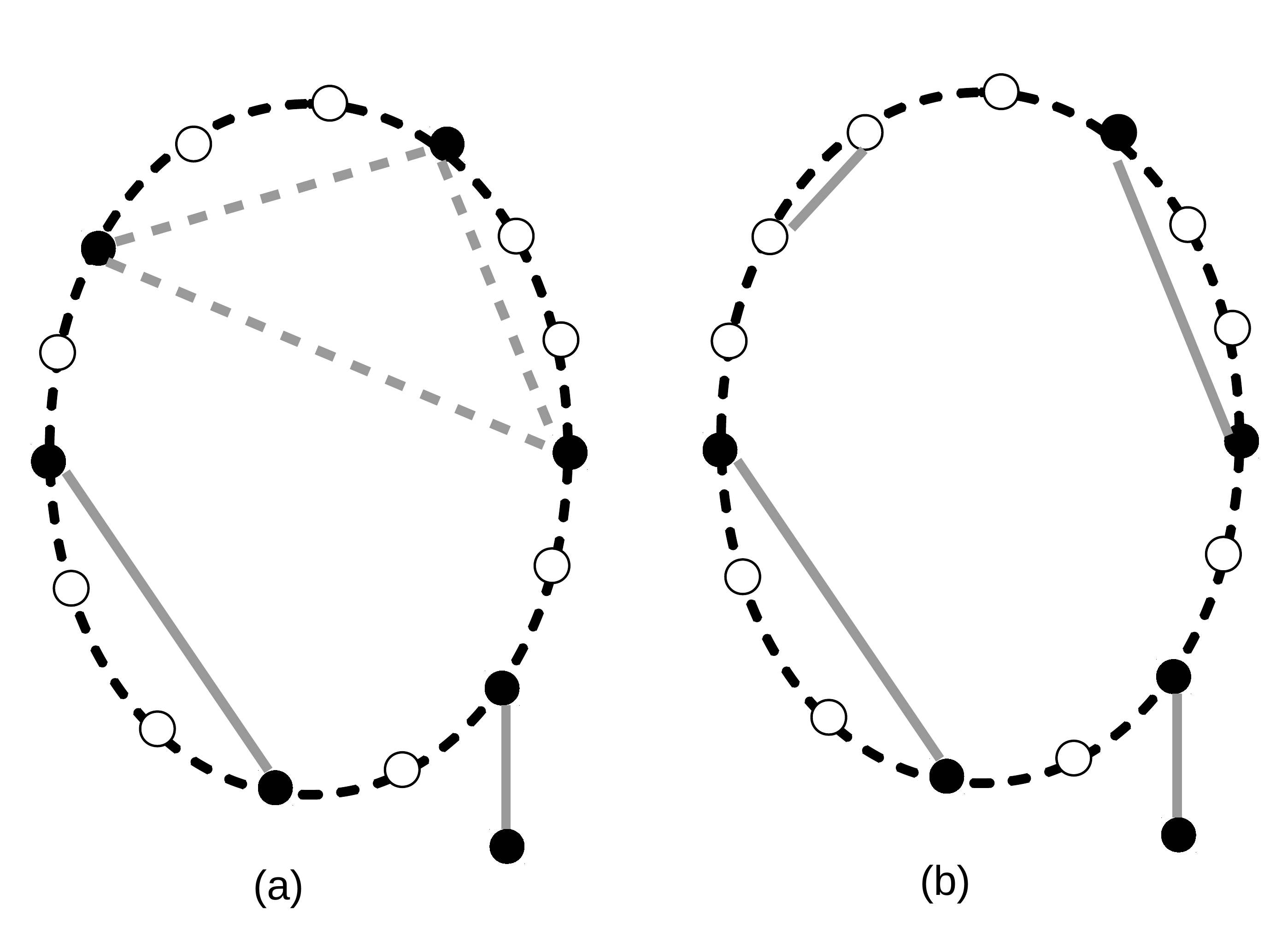}
\end{center}
\caption{A critical cycle is shown in (a). Black edges represent the edges in the support of $y$, and gray edges represent the edges in the support of $\ell$. Continuous lines represent edges of value 1, while dashed lines represent edges of value $\frac{1}{2}$. Nodes in $T(\ell)$ are colored black.
The cycle contains three odd paths whose endpoints are black, and internal nodes are white. All these three paths
satisfy conditions $(i)$ and $(ii)$ of Definition \ref{def:critical}.
The figure in (b) shows how the coordinates of $\ell$ change after performing one move ($\triangleleft$). 
 }
\label{fig:critical} 
\end{figure}

\vspace*{.2cm}
\noindent
{\bf Algorithm 1 (from $z$ to $w$):} \hrulefill\\
1. Set $\ell := z$. 

\vspace*{.1cm}
\noindent
2. While there is a cycle $C \in \mathcal C_{y}\setminus \mathcal C_{\ell}$ that is critical for $\ell$:

\vspace{.2cm}
\hspace{.3cm} 2.1. Let $P \subset C$ be an odd path with nodes $\{v_1, \dots, v_k\}$, such that $V(P) \cap T(\ell) =\{v_1,v_k\}$,

\vspace{.2cm}
\hspace{.3cm} 2.2. If $\{v_1,v_k\} \in \mathcal M_{\ell}$, do the following move:

\vspace{.2cm}
\hspace*{.3cm}
($\triangleright$) Change the $\ell$-coordinate of the edges $\{v_1,v_2\}, \{v_1,v_k\}, \{v_k, v_{k-1}\}$ by setting
\begin{equation*}
\ell_e =
\begin{cases}
1 & \text{if } e = \{v_1,v_2\}  \mbox{ or } e = \{v_k, v_{k-1}\}\\
0 & \text{if } e = \{v_1, v_k\} \\
\end{cases}
\end{equation*}

\hspace{.3cm} and let the nodes $v_2$ and $v_1$ use their tokens to pay for this move;

\vspace{.2cm}
\hspace{.3cm} 2.3. If $v_1,v_k \in V(\bar C)$ for some $\bar C \in \mathcal C_{\ell}\setminus \mathcal C_y$ with $V(\bar C)=\{u_1:=v_1, u_2, \dots, u_{\bar k}\}$, do the move:

\vspace{.2cm}
\hspace*{.3cm}
($\triangleleft$) Change the coordinate $\ell_e$ of all the edges $e \in E(\bar C) \cup \{v_1,v_2\}$ by setting
\begin{equation*}
\ell_e =
\begin{cases}
1 & \text{if } e = \{v_1,v_2\} \\
1 & \text{if } e = \{u_{2i}, u_{2i+1}\} \mbox{ for $ i \in \{1, \dots, \frac{\bar k -1}{2}$\}} \\
0 & \text{if } e = \{u_{2i-1}, u_{2i}\} \mbox{ for $ i \in \{1, \dots, \frac{\bar k -1}{2}$\}}\\
0 & \text{if } e = \{v_{\bar k}, v_1\} \\
\end{cases}
\end{equation*}

\hspace{.3cm} and let the nodes $v_2$ and $v_1$ use their tokens to pay for this move;

\vspace*{.1cm}
\noindent
3. While there is an $\tilde M_{\ell}$-augmenting path $P$ in $\mathcal G_y$, with endpoints $u,v \notin V(\mathcal G_{\ell})$, do the following move:

\vspace{.2cm}
\hspace*{.3cm}
($\circ$) Change the coordinate $\ell_e$ of all the edges of $P$ by setting
\begin{equation*}
\ell_e =
\begin{cases}
1 & \text{if } e \in E(P)\setminus \tilde M_{\ell} \\
0 & \text{if } e \in E(P)\cap  \tilde M_{\ell}  \\
\end{cases}
\end{equation*}
\hspace{.3cm} and let the nodes $u$ and $v$ use their tokens to pay for this move;

\vspace*{.1cm}
\noindent
4. While there is a cycle $C \in \mathcal C_y$ that (i) is packed by $\ell$, and (ii) contains a node $v \in V(C)$ with 
$v \notin \mathcal V(\mathcal G_{\ell})$  do the following move:

\vspace{.2cm}
\hspace*{.3cm}
($\diamond$) Change the coordinates $\ell_e$ of all the edges $e \in E(C)$ to $\frac{1}{2}$,

\vspace{.2cm}
\hspace{.3cm} and let the cycle $C$ and the node $v$ use their tokens to pay for this move.

\vspace*{.1cm}
\noindent
5. Output $w:= \ell$.

\vspace*{-.2cm}
\noindent
\hrulefill

\bigskip
\begin{lemma}
All steps of Algorithm 1 can be correctly performed, and invariants (\ref{eq:invariant}), (\ref{eq:invariant_cycle}), 
 (\ref{eq:invariant2}), and (\ref{eq:invariant3}) are maintained.
\end{lemma}

\begin{proof}
We first prove that Step 2 of algorithm 1 can be correctly performed and maintains all the invariants, by induction on the number of moves executed. 

Let $C$ be a critical cycle for the vertex $\ell$, and $P$ be the corresponding path identified in Step 2.1.
First, suppose that $\{v_1, v_k\} \in \mathcal M_{\ell}$ and let us focus on the move ($\triangleright$).  
We claim that the edges $\{v_1,v_2\}, \{v_1,v_k\}, \{v_k,v_{k-1}\}$ form an 
$\mathcal M_{\ell}$-augmenting path, and therefore the move is a valid move, according to Lemma \ref{cor:adjacency}(b). 
To see this, note first that $v_{k-1} \neq v_2$, since by definition $P$ has at least 3 edges, and so $k\geq 4$. 
Second, we claim that both $v_2$ and $v_{k-1}$ are not in $V(\mathcal G_{\ell})$. 
Assume for a contradiction that $v_2 \in V(\mathcal G_{\ell})$ (the other case is similar).
Since $C$ is critical, we know that $v_2 \notin T(\ell)$, and therefore by invariants (\ref{eq:invariant2}) and (\ref{eq:invariant3}) (which hold by the inductive hypothesis),
$\{v_2,v_3\} \in \tilde M_{\ell}$. Let $\ell'$ be the iteration where $\{v_2,v_3\}$ appears
for the first time in the support of a vertex, and $ \bar \ell$ be the vertex visited by the algorithm
immediately before $\ell'$.  Then, necessarily $C$ was critical for $\bar \ell$ and one of the two nodes
$v_2,v_3$ was in $T( \bar \ell)$. However, if $v_2 \in T( \bar \ell)$, then the path with one single edge $\{v_1,v_2\}$ would contradict
the fact that $C$ was critical for $ \bar \ell$. It follows that $v_3 \in T( \bar \ell)$. 
However, the path $\bar P$ identified at Step 2.1 of the algorithm
has to contain the edge $\{v_2,v_3\}$, and its endpoints are $v_3$ and another node in $T(\bar \ell)$.  
Necessarily, the other endpoint of $\bar P$ is then $v_1$. However, this contradicts the fact that $\bar P$ is odd.
It follows that $\{v_1,v_2\}, \{v_1,v_k\}, \{v_k,v_{k-1}\}$ form an 
$\mathcal M_{\ell}$-augmenting path, and it is easy to see then that
the move ($\triangleright$) maintains the claimed invariants.

Now, suppose that $\{v_1, v_k\} \in V(\bar C)$ for some $\bar C \in \mathcal C_{\ell}\setminus \mathcal C_y$ and let us focus on the move ($\triangleleft$). We claim that $v_2 \notin V(\mathcal G_{\ell})$, and therefore the edge $\{v_1,v_2\}$ and the cycle $\bar C$ have the structure
described in Lemma \ref{cor:adjacency}(f), which allows us to perform a valid move. 
The argument to see this is identical to the one used in the previous point.  
Assume for a contradiction that $v_2 \in V(\mathcal G_{\ell})$.
Since $C$ is critical, we know that $v_2 \notin T(\ell)$, and therefore by invariants (\ref{eq:invariant2}) and (\ref{eq:invariant3}) (which hold by the inductive hypothesis),
$\{v_2,v_3\} \in \tilde M_{\ell}$. Let $\ell'$ be the iteration where $\{v_2,v_3\}$ appears
for the first time in the support of a vertex, and $ \bar \ell$ the vertex visited by the algorithm
immediately before $\ell'$.  Then, necessarily $C$ was critical for $\bar \ell$ and one of the two nodes
$v_2,v_3$ was in $T( \bar \ell)$. However, if $v_2 \in T( \bar \ell)$, then the path with one single edge $\{v_1,v_2\}$ would contradict
the fact that $C$ was critical for $ \bar \ell$. It follows that $v_3 \in T( \bar \ell)$. 
However, the path $\bar P$ identified at Step 2.1 of the algorithm
has then $v_3$ and $v_1$ as its endpoints, but this contradicts the fact that $\bar P$ is odd.
Once again, it is immediate to see that
the move ($\triangleright$) maintains the claimed invariants.

We now argue that the other steps of algorithm 1 can be correctly performed and maintain all the invariants.
Each move in ($\circ$) and ($\diamond$) is indeed a valid move, according to Lemma \ref{cor:adjacency}(b) and Lemma \ref{cor:adjacency}(c), respectively.  It is easy to see that the all invariants are maintained in these steps. 
Furthermore, by invariants (\ref{eq:invariant}) and (\ref{eq:invariant_cycle}), each node $v$ and each cycle $C$ that used a token during the execution of Algorithm 1, satisfy $v \in V(\mathcal G_w)$ and $C \in \mathcal C_w$, and therefore had indeed a token to use.
\end{proof}

We state a trivial observation which will be used later.

\begin{obs}
\label{obs:cycles}
Let $C$  be a cycle in $\mathcal C_{w}\setminus \mathcal C_y$. Then
$C$ did not use its token during the execution of Algorithm 1.
\end{obs}

Before describing the subsequent moves on our path from $w$ to $y$, we need to introduce an important notion: namely,
the notion of \emph{witnesses} of a cycle $C \in \mathcal C_y \setminus \mathcal C_{w}$. This notion will be crucial to identify the nodes
that will pay for the move in which the cycle $C$ appears for the first time in the support of a vertex on our path from  $w$ to $y$.

\subsection{Witnesses}

\begin{definition}
\label{def:witness1}
Let $u$ be a node of a cycle $C \in (\mathcal C_y\setminus \mathcal C_w)$.
We say that $u$ is a \emph{single witness} of $C$ if $u \in T(w)$, $u$ is $(\mathcal M_{w} \cap E(C)$)-exposed, and $C$ is packed by $w$.
\end{definition}

\begin{definition}
\label{def:witness2}
Let $(u,v)$ be a pair of nodes of a cycle $C \in  (\mathcal C_y\setminus \mathcal {C}_{w})$. We say that $(u,v)$ is a \emph{pair of witnesses} of $C$ if $u,v \in T(w)$ and there is an $\tilde M_{w}$-augmenting path 
 $Q(u,v)$ with endpoints $u$ and $v$ in $E(C)$.
\end{definition}

Note that 
for a given pair of witnesses $(u,v)$, the path $Q(u,v)$ is uniquely defined
(since exactly one $u$-$v$ path in $E(C)$ has odd length). 

As already mentioned, we would like that the witnesses of a cycle $C$ pay for the move in which the cycle $C$ appears for the first time in the support of a vertex on our path from the vertex $w$ to the vertex $y$. 
Unfortunately, it might be possible that if we do not select our witnesses carefully, we do not have enough tokens to perform our moves properly. 
We therefore impose some restrictions on our choice.

\begin{definition}
\label{def:set_witnesses}
We say that $\mathcal W \subseteq V \times V$ is a \emph{good set of witnesses} if it satisfies
\begin{itemize}
\item[(i)] for all $C \in \mathcal (C_y\setminus \mathcal {C}_{w})$, the following holds:  $\mathcal W$ contains either exactly one pair of witnesses $(u,v)$ of $C$, or exactly one single witness of $C$ (indicated as $(u,u) \in \mathcal W$);  
\item[(ii)] for every $(u,v) \in \mathcal W$ with $u\neq v$, the following holds: if $|E(Q(u,v))| >1$, then $u$ and $v$ belong to two distinct components of $\mathcal G_{w}$. 
\end{itemize}
\end{definition}

In other words, condition $(ii)$ of the above definition states that if $Q(u,v)$ has more than one edge,
then $\{u,v\} \notin \mathcal M_{w}$, and there is no cycle $\bar C \in \mathcal C_{w}$ such that $V(\bar C)$ contains both $u$ and $v$.
Next lemma shows that there exists a good set of witnesses.

\begin{lemma}
\label{lem:initial_witnesses}
There exists a good set of witnesses $\mathcal W$.
\end{lemma}

\begin{proof} 
It is enough to show that for every $C \in \mathcal (C_y\setminus \mathcal {C}_{w})$, we can find either a single witness, or a pair of witnesses which satisfies the condition in $(ii)$.

Let $C$ be any cycle in $\mathcal C_y\setminus \mathcal C_{w}$. By definition, each edge $e \in \tilde M_{w}$ has its endpoints not in $T(w)$. Therefore, if there is an
edge $\{u,v\} \in E(C)$ such that $u,v \in T(w)$ then the edge $\{u,v\}$ is clearly an 
 $\tilde M_{w}$-augmenting path from $u$ to $v$, and $(u,v)$ are a pair of witnesses for $C$ which 
 satisfies the condition in $(ii)$.

Let us assume there is no edge $\{u,v\} \in E(C)$ such that $u,v \in T(w)$. 
Consider the set $H:=\{v \in V(C): v \in T(w)\}$. 
First, we are going to show that $H \neq \emptyset$.
Suppose otherwise. Then, by invariant (\ref{eq:invariant3}), every node in $v \in V(C)$ satisfies either
$w(\delta(v))=0$, or $w(\delta(v))=1$ and $v$ is an endpoint of an edge in $\tilde M_{w}$.
Note that, by invariant (\ref{eq:invariant2}), if $v$ is an endpoint
of an edge $\{v,\bar v\}$ in $\tilde M_{w}$, then $\bar v$ is also a node of $C$. 
Since $|V(C)|$ is odd but $\tilde M_{w}$ is a matching, it follows that $C$ contains either
an $\tilde M_{w}$-augmenting path whose endpoints are not in $V(\mathcal G_{w})$, or
$C$ is packed by $\mathcal M_{w}$ and one node of $C$ is not in $V(\mathcal G_{w})$. In both cases, Algorithm 1 would not have stopped, a contradiction.

It follows that $H \neq \emptyset$. Consider the connected components $C_1, \dots, C_k$ in the graph $C\setminus H$. Since the nodes in $H$ are pairwise not adjacent in $E(C)$, and $C$ is a cycle, we have that $|V(C)| = |H| + \sum_{i=1}^k |V(C_i)| = k + \sum_{i=1}^k |V(C_i)|$. Since $|V(C)|$ is odd, at least one component $C_i$ has $|V(C_i)|$ equal to an even number. It follows that $C_i$ is a path with an odd number of edges. By Step 3 of Algorithm 1, $\tilde M_{w}$ has to contain a maximum matching of this component. 
It follows that $C_i$ is an $\tilde M_{w}$-alternating path with the first and the last edge belonging to $\tilde M_{w}$.
Let $u$ (resp. $v$) be the node in $H$ adjacent to the first (resp. last) edge of the path $C_i$ in the graph $C$. 
If $u=v$ (i.e. $k=1$), then $C$ is packed by $w$, and
$u$ is a single witness for $C$.
 If instead $u \neq v$, then $u, C_i, v$ yield an $\tilde M_{w}$-augmenting path from $u$ to $v$ in the subgraph induced by $\tilde M_{w} \cup E(C)$. 
 
The above argument shows that for every cycle $C \in \mathcal C_{w}\setminus \mathcal C_{y}$ there exists either a single witness,
or (at least) one pair of witnesses. It remains to show that, if $C$ is not packed by $w$, 
then among all possible pairs of nodes which satisfy the definition of pair of witnesses,
at least one satisfies the condition described in $(ii)$. 
For the sake of a contradiction, assume that no pairs of witnesses satisfies the condition described in $(ii)$.
Then, this means that $C$ is critical for the vertex $w$. Let $\bar \ell$ be the last vertex visited by Algorithm 1
during the execution of Step 2. Clearly, $C$ cannot be critical for $\bar \ell$, otherwise the algorithm would have performed 
another iteration of Step 2. However, note that for all vertices $\ell'$ visited after
$\bar \ell$ in Step 3 and Step 4, we have $T(\bar \ell)=T(\ell')$, and therefore
$C$ cannot be critical for any $\ell'$. Since $w$ is also visited after $\bar \ell$,
$C$ cannot be critical for $w$, a contradiction.
\end{proof}

From now on, we fix $\mathcal W$ to be a good set of witnesses (one such set exists because of last lemma),
and based on that, we introduce the last two ingredients needed to describe our future moves, 
namely, \emph{target matchings}, and the \emph{target graph}.

\subsection{Target matchings and target graph}

\begin{definition}
Let $C \in \mathcal C_y\setminus \mathcal C_{w}$.    
The \emph{target matching} of $C$, denoted by $M_C$, is a matching which satisfies the following properties:
\begin{itemize}
\item[(i)] $M_C$ is a maximum cardinality matching of $C$;
\item[(ii)] If $(u,v) \in \mathcal W$ is a pair of witnesses for $C$ with $u\neq v$, then $M_C \cap E(Q(u,v))$ is a perfect matching of $Q(u,v)$;
\item[(iii)] among all matchings satisfying (i) and (ii),
$M_C$ maximizes the quantity $|M_C \cap \mathcal M_{w}|$.
\end{itemize}
\end{definition}

We refer to Figure~\ref{fig:witness} for an example.

\begin{figure}
\begin{center} 
\includegraphics[angle=0, width=8cm, height=5cm]{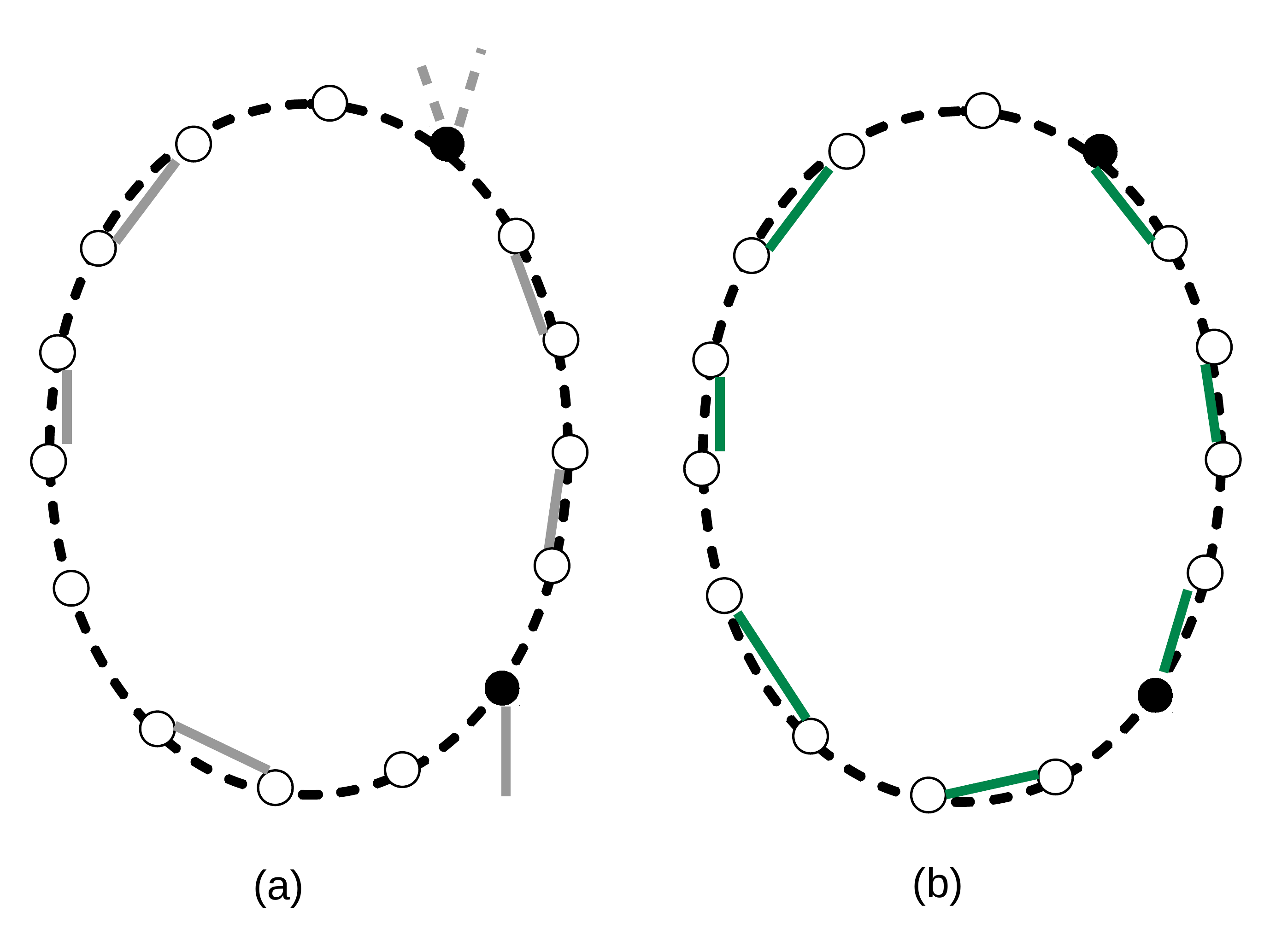}
\end{center}
\caption{A cycle $C$ in $\mathcal C_y \setminus \mathcal C_w$ is shown in (a). Black edges represent the edges in the support of $y$, and gray edges represent the edges in the support of $w$. Continuous lines represent edges of value 1, while dashed lines represent edges of value $\frac{1}{2}$. 
Black nodes represent the witnesses of $C$. The figure in (b) shows the witness matching (dark green edges). 
 }
\label{fig:witness} 
\end{figure}

\begin{definition}
Let $\ell$ be a vertex. The \emph{target graph} $\mathcal T^{\ell}$ is the graph induced by the edges
$$\mathcal M_{\ell} \Delta (\mathcal M_y \cup \{\bigcup_{C \in \mathcal C_{y}\setminus \mathcal C_{\ell}} M_C \})$$ 
\end{definition}

\noindent
Note that the target graph is the symmetric difference of two matchings, namely $\mathcal M_{\ell}$, and $\mathcal M_y \cup \{\bigcup_{C \in \mathcal C_{y}\setminus \mathcal C_{\ell}} M_C \}$, and therefore it is the disjoint union of paths and even cycles.
We call a component $K$ of $\mathcal T^{\ell}$ a \emph{path}-component if $K$ is a path, 
and a \emph{cycle}-component if $K$ is a cycle.

Roughly speaking, our goal is to move from $w$ to a vertex whose support graph does not contain any cycle in $\mathcal C_w \setminus C_y$,
by performing a sequence of moves, each involving one component of the target graph.
However, we would like not to use tokens belonging to witness nodes to pay for the moves, since as already
mentioned, we would like to keep these tokens to pay for the moves where cycles in $\mathcal C_y \setminus \mathcal C_{w}$
show up. For this reason we introduce the following definition.

\begin{definition}
Let $\ell$ be a vertex. A component $K$ of a target graph $\mathcal T^{\ell}$ is called 
\emph{dangerous} if all the nodes of $V(K)\cap T(\ell)$ 
are witnesses of some cycles in $\mathcal C_y \setminus \mathcal C_{\ell}$.
\end{definition}

Next lemmas give a few properties that will be crucial for our analysis. 
\begin{lemma}
\label{lem:property_target_graph}
Let $K$ be a component of the target graph $\mathcal T^{w}$. Then
\begin{itemize}
\item[(i)] $V(K) \cap T(w) \neq \emptyset$;
\item[(ii)] If $K$ is a dangerous cycle-component, then there exist at least two distinct cycles in $\mathcal C_w\setminus \mathcal C_y$
whose witnesses are in $V(K)$;
\item[(iii)] If $K$ is a dangerous path-component, then each endpoint $v$ of $K$ is either a single witness in $\mathcal W$, 
or $v \in V(C)$ for some cycle $C \in \mathcal C_{w}\setminus C_{y}$. 
\end{itemize}
\end{lemma}
\begin{proof} 
Let $K$ be a component of the target graph $\mathcal T^{w}$. We first prove (i). 
Assume $V(K) \cap T(w) = \emptyset$. Then $E(K) \cap \mathcal M_{w} \subseteq \tilde M_{w}$, 
by invariant (\ref{eq:invariant3}).
Since $\tilde M_{w} \subseteq E(\mathcal G_{y})$ by invariant (\ref{eq:invariant2}), it follows that
$K \subseteq \mathcal G_y$, and therefore $K$ is a path.
If $K$ has even length, then $K \subset C$ for some $C \in \mathcal C_y$, and it
is an $M_C$-alternating path. 
By switching along the edges of this path we could get
another matching $M'_{C}$ which would contradict our choice of the witness matching
for $C$, since $M'_{C}$ would satisfy (i) and (ii), but $|M'_C \cap \tilde M_{w}| > |M_C \cap \tilde M_{w}|$.
It follows that $K$ has odd length, and since $\mathcal M_y \cup \{\cup_{C \in \mathcal C_y} M_C\}$ is
a maximum matching in $\mathcal G_{y}$, $K$ has to be $\tilde M_{w}$-augmenting. 
Then, if both endpoints $c',c$ of $K$ are not in $T(w)$, then $w(\delta(c'))=w(\delta(c))=0$
(since otherwise they would still have their token available, being non-endpoints of edges in $\tilde M_w$, by invariant 
(\ref{eq:invariant3})). However, this contradicts the termination of Step 3 of Algorithm 1.
It follows that at least one of the endpoints has to be in $T(w)$.

Now we prove (ii). Let $K$ be a dangerous cycle-component, and $v$ be a node in $ V(K)\cap T(w)$ (such node exists because of (i)). 
Since $K$ is dangerous, 
$v$ is a witness of a cycle $C \in \mathcal C_y$. 
Since $K$ is a cycle, $v$ cannot be a single witness (such nodes
have degree at most 1 in the target graph). It follows that
there is another node $u$ such that $(u,v)$ is a pair of witnesses for $C$. 
Let $v'$ (resp. $u'$) be the node such that $\{v,v'\} \in \mathcal M_{w}$
(resp. $\{u,u'\} \in \mathcal M_{w}$). Note that $v'\neq u$ (resp. $u'\neq v$),
since otherwise $\mathcal W$ would not be a good set of witnesses (the pair $(u,v)$ would contradict
the second condition of Definition \ref{def:set_witnesses}).
By invariant (\ref{eq:invariant2})
both $u',v' \in T(w)$, and therefore, since $K$ is dangerous, they are (in a pair of) witnesses for at least another cycle
in $\mathcal C_y$.

Finally we prove (iii). Let $K$ be a dangerous path-component, with $V(K):=\{v_1, \dots,  v_k\}$. 
Let $i$ be the smallest index such that $ v_i \in T(w)$ (such index exists because of (i)). 
Since $K$ is dangerous, $v_i$ is a witness (either a single one, or a paired one) of a cycle $C \in \mathcal C_y$. 
If $v_i$ is a single witness, then it is necessarily the endpoint of $K$, 
i.e. $i=1$, since single witnesses have degree at most one in the target graph.
Assume now that  $v_i$ is in a pair of witnesses. In this case, $v_i$ cannot be an endpoint of an edge in $\mathcal M_{w}$:
using invariant (\ref{eq:invariant3}), we can see that
if $\{v_i, v_{i+1}\} \in \mathcal M_{w}$, then necessarily $v_{i+1}$ is the other witness 
of $C$, i.e. $Q(v_i, v_{i+1}) = \{v_i, v_{i+1}\}$ contradicting
that this edge is in the target graph,  and if $\{v_i, v_{i-1}\} \in \mathcal M_{w}$, our choice of $i$ is contradicted. It follows that $v_i \in V(\bar C)$ for some
$\bar C \in \mathcal C_w\setminus \mathcal C_y$. This implies that 
$v_i$ have degree one in the target graph, and therefore it is 
the endpoint of $K$ (i.e. $i=1$).
We can apply the same argument to the biggest index such that $ v_i \in T(w)$, and get the same conclusion
for the other endpoint of $K$.
\end{proof}

\subsection{Moving from $w$ to $r$}
We move from $w$ to a vertex $r$ with the property that $\mathcal C_{r} \subseteq
\mathcal C_y$, by eliminating up to two cycles in $\mathcal C_w\setminus \mathcal C_y$ at each move. 

The algorithm maintains six invariants for every vertex $\ell$ visited during its execution. The first three invariants
guarantee that the conditions of Lemma \ref{lem:property_target_graph} hold for every vertex visited by the algorithm, i.e.

\begin{equation}
\label{eq:invariant_1bis}
\mbox{For every component $K$ of } \mathcal T^{\ell}, \; V(K) \cap T(\ell) \neq \emptyset
\end{equation}

\begin{equation}
\label{eq:invariant_2bis}
\begin{array}{c}
\mbox{If } K \mbox{ is a dangerous cycle-component of } \mathcal T^{\ell}, \mbox{ then }
V(K) \mbox{ contains the witnesses} \\ \mbox{ of at least two cycles in } \mathcal C_{\ell}\setminus \mathcal C_y 
\end{array}
\end{equation}

\begin{equation}
\label{eq:invariant_3bis}
\begin{array}{c}
\mbox{If } K \mbox{ is a dangerous path-component of } \mathcal T^{\ell},
\mbox{ then each endpoint $v$ of $K$ is} \\ \mbox{either a single witness in $\mathcal W$,
 or $v \in V(C)$ for some cycle $C \in \mathcal C_{\ell}\setminus C_{y}$}
\end{array}
\end{equation}

\smallskip
\noindent
The fourth and fifth invariants will be useful to keep track of which nodes in $\mathcal T^\ell$ still have their tokens. 
The fourth invariant states that for every edge of $\mathcal M_{\ell}$ which is present in the target graph,  either both the endpoints have the token, or both the endpoints used their tokens.

\begin{equation}
\label{eq:invariant_4bis}
\forall \{u,v\} \in \mathcal M_{\ell} \cap E(\mathcal T^{\ell}), \; \mbox{ if $u \in T(\ell)$ then $v \in T(\ell)$ (and vice versa)}
\end{equation}
\noindent
Note that (\ref{eq:invariant_4bis}) holds for $w$, because of invariant (\ref{eq:invariant3}) maintained by Algorithm 1.

The fifth invariant states that every edge of $\tilde M_{\ell}$ that is present in the target graph, is an edge of $E(\mathcal C_{y})$.

\begin{equation}
\label{eq:invariant_5bis}
\forall \{u,v\} \in \tilde M_{\ell}  \cap E(\mathcal T^{\ell}), \; \{u,v\} \in E(\mathcal C_{y})
\end{equation}

\noindent
Invariant (\ref{eq:invariant_5bis}) holds for $w$, because of invariant (\ref{eq:invariant2}) maintained by Algorithm 1,
together with the fact that if $e \in \tilde M_{\ell} \cap \mathcal M_y$, then $e \notin E(\mathcal T^\ell)$.

\smallskip
\noindent
Finally, the last invariant establishes that the witnesses of cycles in $\mathcal C_y \setminus \mathcal C_{\ell}$ have their token
available.

\begin{equation}
\label{eq:invariant_6bis}
\forall C \in \mathcal C_y \setminus \mathcal C_{\ell}, \mbox{ if } v\in V(C) \mbox{ is a witness of } C,  \mbox{ then } v \in T(\ell).
\end{equation}

\noindent
Clearly invariant (\ref{eq:invariant_6bis}) holds for $w$, because of the definition of witnesses.

\smallskip
The algorithm selects one cycle in $C \in  \mathcal C_{\ell}\setminus \mathcal {C}_y$ at the time,
and performs a move which involves at most one path-component of the target graph. 
Note that $C$ intersects each component of $\mathcal T^{\ell}$ in at most two nodes. To see this, recall that each component $K$ of $\mathcal T^{\ell}$  is either an $\mathcal M_{\ell}$-alternating path, or an $\mathcal M_{\ell}$-alternating cycle. 
Therefore, $C$ can only intersect $K$ at the endpoints of an $\mathcal M_{\ell}$-alternating path. Since
$C$ has an odd number of nodes, it follows that there exists 
either one node $v$ in $V(C)$ which is not a node of $\mathcal T^{\ell}$,
or one component $K$ of $\mathcal T^{\ell}$ such that $|V(C) \cap V(K)| =1$.
Based on this observation, we have the following definition.

\begin{definition}
Let $C \in  (\mathcal C_{\ell}\setminus \mathcal {C}_y)$. 
The \emph{least-intersecting} component for $C$ is a component
$H$ of the graph $(V(\mathcal T^{\ell} \cup C), E(\mathcal T^{\ell}))$
that minimizes $|V(H) \cap V(C)|$.
\end{definition}

Note that the graph $(V(\mathcal T^{\ell} \cup C), E(\mathcal T^{\ell}))$
is simply the graph $\mathcal T^{\ell}$ with (possibly) additional singleton nodes
that belong to $V(C)$.
By the above reasoning, a least-intersecting component $H$ for $C$
will always have $|V(H) \cap V(C)| =1$. In particular, 
either $H$
will be path-component of $\mathcal T^{\ell}$ with exactly one endpoint
in $V(C)$, or $H$ will be a singleton node $v$ for some $v \in V(C)$
(this can happen if there exists a node $v \in V(C)$ with $v \notin V(\mathcal T^{\ell})$).
Algorithm 2 formally describes the moves we perform to go from $w$ to $r$.

\vspace*{.2cm}
\noindent
{\bf Algorithm 2 (from $w$ to $r$):} \hrulefill\\
1. Set $\ell := w$. 

\vspace*{.1cm}
\noindent
2. While there exists a cycle $C \in \mathcal C_{\ell}\setminus C_y$ do:
\begin{enumerate}
\item[2.1] Let $H$ be the least-intersecting component for $C$, and $V(H):=\{v_1, \dots, v_k\}$ with $v_1 \in V(C)$
\item[2.2] If $v_k$ is a single witness for some $\bar C \in \mathcal C_{y}\setminus \mathcal C_{\ell}$, 
let $N$ be the perfect matching of $C$ exposing $v_1$ and do the following move:

\vspace{.2cm}
\hspace*{.3cm}
($\triangle$) Change the coordinate $\ell_e$ of all the edges of $H \cup C \cup \bar C$ by setting
\begin{equation*}
\ell_e =
\begin{cases}
1 & \text{if } e \in E(H)\setminus \mathcal M_{\ell} \\
0 & \text{if } e \in E(H)\cap \mathcal M_{\ell}  \\
1 & \text{if } e \in N \\
0 & \text{if } e \in E(C)\setminus N  \\
\frac{1}{2} & \text{if } e \in E(\bar C) \\
\end{cases}
\end{equation*}
\hspace{.3cm} and let the node $v_k$ and the cycle $C$ use their tokens to pay for this move;

\item[2.3] Else if $v_k \in V(\bar C)$ for some $\bar C \in \mathcal C_{\ell}\setminus C_y$, $\bar C \neq C$, 
let $N$ be the perfect matching of $C$ exposing $v_1$, let $\bar N$ be the perfect matching of $\bar C$ exposing $v_k$, and do the following move:

\vspace{.2cm}
\hspace*{.3cm}
($\triangledown$) Change the coordinate $\ell_e$ of all the edges of $H \cup C \cup \bar C$ by setting
\begin{equation*}
\ell_e =
\begin{cases}
1 & \text{if } e \in E(H)\setminus \mathcal M_{\ell} \\
0 & \text{if } e \in E(H)\cap \mathcal M_{\ell}  \\
1 & \text{if } e \in N \\
0 & \text{if } e \in E(C)\setminus N  \\
1 & \text{if } e \in \bar N \\
0 & \text{if } e \in E(\bar C)\setminus \bar N  \\
\end{cases}
\end{equation*}
\hspace{.3cm} and let the cycles $C$ and $\bar C$ use their tokens to pay for this move;

\item[2.4] Else, let $j$ be 
any index such that $v_j \in V(H) \cap T(\ell)$ and $v_j$ is not a witness in $\mathcal W$, let $N$ be the perfect matching of $C$ exposing $v_1$, and do the following move:

\vspace{.2cm}
\hspace*{.3cm}
($\square$) Change the coordinate $\ell_e$ of all the edges of $H \cup C$ by setting
\begin{equation*}
\ell_e =
\begin{cases}
1 & \text{if } e \in E(H)\setminus \mathcal M_{\ell} \\
0 & \text{if } e \in E(H)\cap \mathcal M_{\ell}  \\
1 & \text{if } e \in N \\
0 & \text{if } e \in E(C)\setminus N  \\
\end{cases}
\end{equation*}
\hspace{.3cm} and let the node $v_j$ and the cycle $C$ use their tokens to pay for this move;

\end{enumerate}

\vspace*{.1cm}
\noindent
3. Output $r:={\ell}$. 

\vspace*{-.2cm}
\noindent
\hrulefill

\bigskip
\begin{lemma}
All steps of Algorithm 2 can be correctly performed, and invariants (\ref{eq:invariant_1bis}), (\ref{eq:invariant_2bis}),
(\ref{eq:invariant_3bis}), (\ref{eq:invariant_4bis}), (\ref{eq:invariant_5bis}) 
and (\ref{eq:invariant_6bis}) are maintained.
\end{lemma}

\begin{proof}
First we argue that the moves are indeed valid moves. 
The move
 ($\triangle$) is a valid move according to either Lemma \ref{cor:adjacency}(d) (if $v_k=v_1$) or 
 Lemma \ref{cor:adjacency}(e) (if $v_k \neq v_1$).
 The move
 ($\triangledown$) is a valid move according to
 Lemma \ref{cor:adjacency}(e).
Furthermore, the move  
  ($\square$) is a valid move according to 
  either Lemma \ref{cor:adjacency}(c) (if $v_k=v_1$), or 
  Lemma \ref{cor:adjacency}(f) (if $v_k \neq v_1$).

Second, we argue that the invariants hold, by induction
on the number of iterations of Step 2 performed by the algorithm.
Let $C$ be a cycle in $\mathcal C_{\ell}\setminus \mathcal C_y$ considered  at a Step 2 of the algorithm, and $H$
the component in Step 2.1, with nodes $\{v_1, \dots, v_k\}$ (with possibly, $v_k=v_1$).
Let $\bar \ell$ be the vertex visited by the algorithm right after $\ell$. 
Invariants (\ref{eq:invariant_1bis}), (\ref{eq:invariant_4bis}) and (\ref{eq:invariant_5bis}) follow by two things: (i)  
the nodes paying for the moves are in $V(H)$, but $V(H)$ and $E(H)$ 
will not be present in $\mathcal T^{\bar \ell}$, (ii) 
if $K$ is a component of $\mathcal T^{\bar \ell}$ but not a component
of $\mathcal T^{\ell}$, then $K$ contains at least one edge $e \in \mathcal M_{\bar \ell}\setminus \mathcal M_{\ell}$.
Necessarily, the endpoints of $e$ are in $V(C)\setminus \{v_1,v_k\}$, i.e. $e \notin \tilde M_{\ell}$, 
and therefore still have their token available. 

Let $K$ be a dangerous cycle-component of $\mathcal T^{\bar \ell}$.
If $K$ is also a dangerous cycle-component of $\mathcal T^{\ell}$, then 
the condition of invariant (\ref{eq:invariant_2bis}) clearly still hold. 
Otherwise, $K$ contains at least one edge $e = \{u,\bar u\}$ in $\mathcal M_{\bar \ell} \setminus \mathcal M_{\ell}$.
By the same reasoning as above, $u$ and $\bar u$ are in $T(\bar \ell)$, and since $K$ is dangerous,
they must be witnesses in $\mathcal W$. However, $u$ and $\bar  u$ cannot be 
a pair of witnesses for the same cycle $\bar C \in \mathcal C_y$, 
since otherwise this would contradict condition $(ii)$ of Definition \ref{def:set_witnesses}.
It follows that invariant (\ref{eq:invariant_2bis}) holds.

The argument 
for invariant (\ref{eq:invariant_3bis}) is identical to the one used in Lemma \ref{lem:property_target_graph}.
We repeat it here for the sake of completeness.
Let $K$ be a dangerous path-component, with $V(K):=\{u_1, \dots,  u_k\}$. 
Let $i$ be the smallest index such that $ u_i \in T(\bar \ell)$ (such index exists because of invariant (\ref{eq:invariant_1bis})). 
Since $K$ is dangerous, $u_i$ is a witness (either a single one, or a paired one) of a cycle $\bar C \in \mathcal C_y$. 
If $u_i$ is a single witness, then it is necessarily the endpoint of $K$, 
i.e. $i=1$, since single witnesses have degree at most one in the target graph.
Assume now that  $u_i$ is in a pair of witnesses. In this case, $u_i$ cannot be an endpoint of an edge in $\mathcal M_{\bar \ell}$:
using invariant (\ref{eq:invariant_4bis}), we can see that
if $\{u_i, u_{i+1}\} \in \mathcal M_{\bar \ell}$, then necessarily $u_{i+1}$ is the other witness 
of $\bar C$, i.e. $\{u_i, u_{i+1}\} = Q(u_i, u_{i+1})$ contradicting
that this edge is in the target graph,  and if $\{u_i, u_{i-1}\} \in \mathcal M_{\bar \ell}$, our choice of $i$ is contradicted. It follows that $u_i \in V(\tilde C)$ for some
$\tilde C \in \mathcal C_{\bar \ell}\setminus \mathcal C_y$. This implies that 
$u_i$ has degree one in the target graph, and therefore it is 
the endpoint of $K$ (i.e. $i=1$).
We can apply the same argument to the biggest index such that $ u_i \in T(\bar \ell)$, and get the same conclusion
for the other endpoint of $K$.

Invariant (\ref{eq:invariant_6bis}) holds trivially, since if we use a token of a witness node $v_k$
to pay for moving from $\ell$ 
 to $\bar \ell$, then $v_k$ was a single witness and the cycle witnessed by $v_k$ is in $\mathcal C_{\bar \ell} \cap \mathcal C_y$.
 
Finally, we argue that there are enough tokens to pay for the moves. 
In Step 2.2, the cycle $C$ has a token by Observation \ref{obs:cycles}, 
and $v_k$ has a token by invariant (\ref{eq:invariant_5bis}).
In Step 2.3, the cycles $C$ and $\bar C$ have their tokens by Observation \ref{obs:cycles}. 
In Step 2.4, the cycle $C$ has a token by Observation \ref{obs:cycles}, 
and $v_j$ exists: invariant (\ref{eq:invariant_1bis}) 
guarantees that $V(H) \cap T(\ell)$ is not empty, and
$H$ cannot be dangerous, otherwise by
invariant (\ref{eq:invariant_3bis}) we would perform
the operation in Step 2.3 or 2.2.
\end{proof}
\noindent
We conclude this section with a lemma that lists some useful properties
satisfied by vertex $r$.

\begin{lemma}
\label{lem:r_properties}
The vertex $r$ satisfies the following properties:
\begin{itemize}
\item[(a)] Each component of the target graph $\mathcal T^r$
contains at least two nodes that are in $T(r)$.
\item[(b)] Let $C\in \mathcal C_y \setminus \mathcal C_{r}$ be a cycle with a single witness $(u,u) \in \mathcal W$.
Then $C$ is packed by $r$, and $u$ is the endpoint of some path-component of $\mathcal T^r$.
\item[(c)] Let $C\in \mathcal C_y \setminus \mathcal C_{r}$ be a cycle with a pair of witnesses $(u,v) \in \mathcal W$, $u\neq v$. 
If $u \in V(K)$ for some component $K$ of $\mathcal T^r$, then $Q(u,v) \subseteq K$. Furthermore,
$u$ (resp. $v$) is incident into at least two edges of $K$.
 \item[(d)] Let $C\in \mathcal C_y \setminus \mathcal C_{r}$, 
 and $K$ be a dangerous cycle-component of $\mathcal T^r$
 such that $|V(C) \cap V(K)| \neq \emptyset$. Then $K$ contains the witnesses of $C$. 
\end{itemize}
\end{lemma}

\begin{proof}
(a) follows since each component of the target graph has at least one node in $T(r)$ by invariant (\ref{eq:invariant_1bis}).
This node is the endpoint of some edge $e$ in $\mathcal M_{r}$, and the other endpoint of $e$ also is in $T(r)$ by invariant (\ref{eq:invariant_4bis}).

Let us now argue about (b). Let $C\in \mathcal C_y \setminus \mathcal C_{r}$ be a cycle with a single witness $(u,u) \in \mathcal W$. Then
$C$ was packed by $w$. This implies that all the edges of $\mathcal M_w \cap E(C)$ were not present in any target graph
$\mathcal T^{\ell}$, for any $\ell$ visited by Algorithm 2, and therefore $\mathcal M_w \cap E(C) = \mathcal M_r \cap E(C)$,
implying that $C$ is packed by $r$. Since $u$ is in $T(r)$ by invariant (\ref{eq:invariant_6bis}),  
$u$ is an endpoint of an edge in $\mathcal M_r$, i.e. $u$ belongs to a component of $\mathcal T^r$.

To see (c), let $C\in \mathcal C_y \setminus \mathcal C_{r}$ be a cycle with a pair of witnesses $(u,v) \in \mathcal W$, $u\neq v$, and
let $u$ be a node of a component $K$ of  $\mathcal T^r$. By the definition of target matching, $u$ is incident into an edge $e$ of the target matching
$M_C$. By invariant (\ref{eq:invariant_6bis}), $u$ is in $T(r)$ and so $u$ is an endpoint of an edge in $\mathcal M_r$, 
implying that $u$ is incident into two edges of $K$. Furthermore, by definition of target matching, $e$ is an edge of $Q(u,v)$,
and $Q(u,v)$ is an $\tilde M_w$-augmenting path. This implies that $Q(u,v)$ is contained in a component of $\mathcal T^{w}$, 
and that we never performed a move involving $e$ during the execution of Algorithm 2 (otherwise, one can see that $e$ would be in $M_C \cap \mathcal M_{r}$, implying $e \notin E(\mathcal T^r)$). It follows that $Q(u,v)$ is contained a component of $\mathcal T^{\ell}$ for all vertices $\ell$ visited by Algorithm 2 up to $r$, and
in particular, $Q(u,v)$ is contained the component $K$ of $\mathcal T^{r}$. By the same argument we used for $u$, $v$ is incident into two edges of $K$.

Finally, we prove (d). Assume that $K$ does not contain any witness of $C$.
Let $e \in E(K)$ be an edge with (at least) one endpoint $u$ in $V(C)$, such that $e \notin E(C)$. Note that such edge must exist, 
since $C$ is an odd cycle, and $K$ is an even cycle. Necessarily $e \in \mathcal M_{r}$. 
However, $u$ is not a witness of $C$, and since $K$ is dangerous, $u \notin T(r)$.
By invariant (\ref{eq:invariant_4bis}), both endpoints of $e$ are not in $T(r)$, i.e. $e \in \tilde M_r$.  
However, this contradicts invariant (\ref{eq:invariant_5bis}).
\end{proof}

\subsection{Moving from $r$ to $y$} Given $r$ from the previous procedure, we move to $y$ using Algorithm 3. 
In Step 2, we perform moves involving path-components of the target graph, and
in Step 3, we perform moves involving non-dangerous cycle-components of the target graph. In both cases, 
each component has (at least) two nodes with an available token to pay for the move. 
In particular, for a path-component,
if any of these nodes is a single witness of some cycle $C \in \mathcal C_y\setminus \mathcal C_{\ell}$,  
then the cycle will appear in the support of next vertex.

Dangerous cycle-components are considered last, in Step 4,
and for each such component $K$, we perform two moves.
Roughly speaking, a move which switches
the coordinate values along the edges of $E(K)$ would require
two tokens which are witnesses of some cycle $C \in \mathcal C_y\setminus \mathcal C_{\ell}$, but 
differently from dangerous path-components, here 
such move will not make $C$ appearing in the support of next vertex.
Therefore, we first make a move which guarantees that two distinct
cycles $C$ and $\bar C$ becomes packed, and then perform a second move which  
guarantees that both $C$ and $\bar C$ appear in the support of our vertex.
We use the tokens of the two witnesses of $C$ to pay for the first move, 
and the tokens of the two witnesses of $\bar C$ to pay for the second move.

To describe our moves formally, we introduce one definition (see Figure~\ref{fig:dang} for an illustration).

\begin{definition}
\label{def:adjacency}
Let $K$ be a dangerous cycle-component of $\mathcal T^{\ell}$, and let $C,\bar C \in \mathcal C_y\setminus \mathcal C_{\ell}$ 
be two distinct cycles with $V(C) \cap V(K) \neq \emptyset$ and $V(\bar C) \cap V(K) \neq \emptyset$. Let
$(u,v)$ (resp. $(\bar u, \bar v)$) be the pair of witnesses of $C$ (resp. $\bar C$).
We call a path $P$ a \emph{$K$-linking path} between $C$ and $\bar C$ if:
\begin{itemize}
\item[(i)] $P$ is an $\mathcal M_{\ell}$-augmenting path; 
\item[(ii)] $E(P) \cap E(K) = \{v,\bar v\}$,
\item[(iii)] for all $e' \in E(P)$ with $e' \neq \{v,\bar v\}$, we have $e' \in E(C\setminus Q(u,v)) \cup E(\bar C\setminus Q(\bar u, \bar v))$.
\end{itemize}
\end{definition}

\begin{figure}[H]
\begin{center} 
\vspace{-.5cm}
\includegraphics[angle=0, width=9cm, height=7cm]{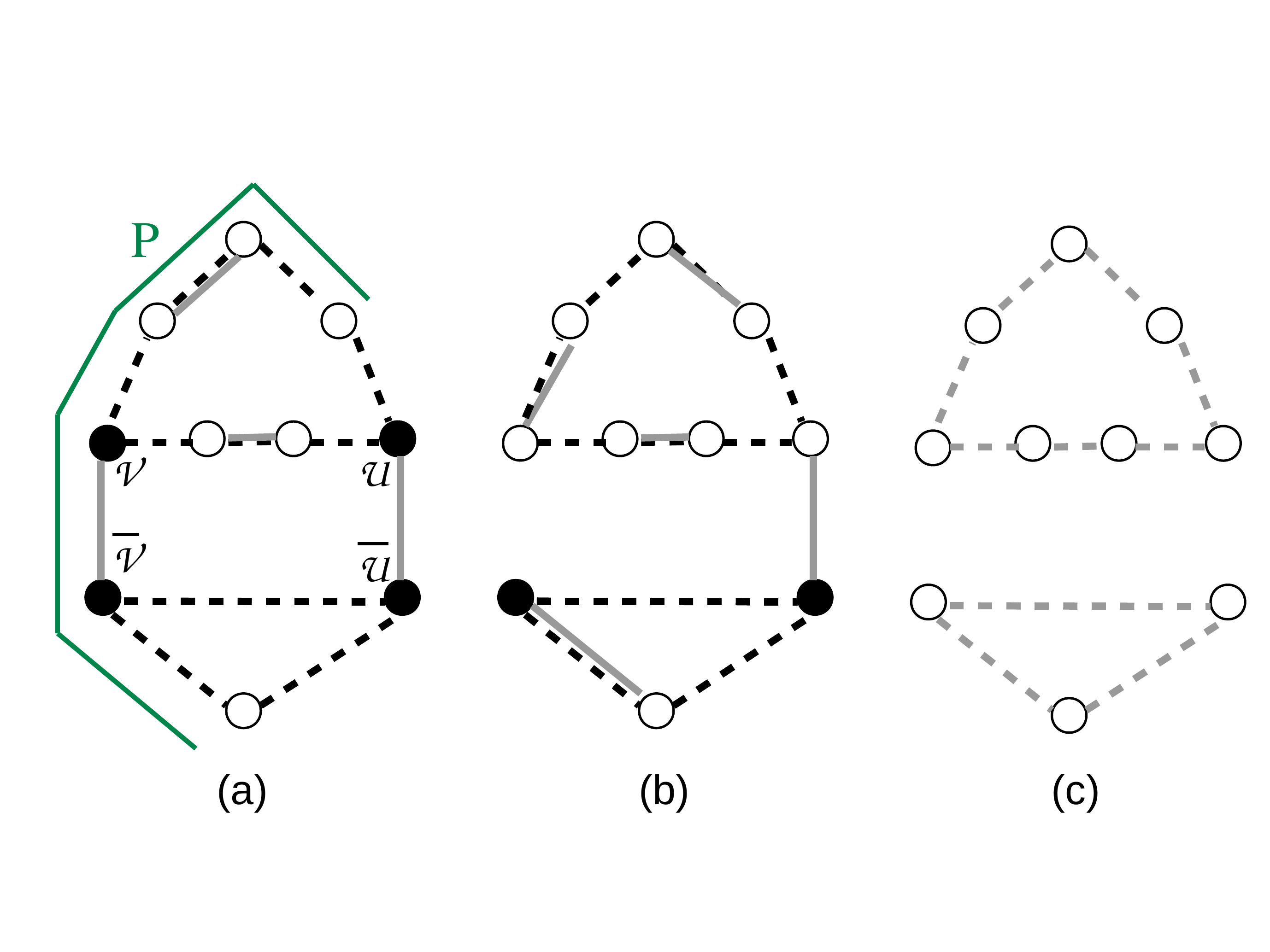}
\end{center}
\vspace{-1cm}
\caption{Two cycles in $\mathcal C_y \setminus \mathcal C_{\ell}$ intersecting a dangerous component are shown in (a).
Black edges represent the edges in the support of $y$, and gray edges represent the edges in the support of $\ell$. Continuous lines represent edges of value 1, while dashed lines represent edges of value $\frac{1}{2}$. 
Black nodes represent nodes in $T(\ell)$. 
The dangerous component is the even cycle given by the edges $\{v,\bar v\}, \{u,\bar u\}$, the odd $u$-$v$ path ($Q(u,v)$)
contained in the top dashed black cycle, and the  odd $\bar u$-$ \bar v$ path ($Q(\bar u, \bar v)$)
contained in the bottom dashed black cycle. The dark green line indicates the path $P$ as in Definition
\ref{def:adjacency}.
The figure in (b) shows how the coordinates of $\ell$ change after performing the move ($\otimes$),
while the figure in (c) shows how the coordinates of $\ell$ change after performing the move ($\oslash$).  
 }
\label{fig:dang} 
\end{figure}

\noindent
In Step 5, we perform our last set of moves, which make all remaining cycles in 
$\mathcal C_y\setminus \mathcal C_{\ell}$ appear, one by one, in the support of next vertex.

For any vertex $\ell$ visited by Algorithm 3, we will maintain invariant (\ref{eq:invariant_6bis})
as well as the following invariant:
\begin{equation}
\label{eq:invariant_final}
\mbox{If $K$ is a component of } \mathcal T^{\ell}, \mbox{ then $K$ is a component of } \mathcal T^{r} \mbox{ and } V(K)\cap T(\ell) = V(K) \cap T(r)
\end{equation}

\vspace*{.2cm}
\noindent
{\bf Algorithm 3 (from $r$ to $y$):} \hrulefill\\
1. Set $\ell := r$.

\vspace*{.1cm}
\noindent
2. While there exists a path-component $K$ of $\mathcal T^{\ell}$:
\begin{enumerate}
\item[2.1] Let $V(K) :=\{v_1, \dots, v_k\}$, let $E(K) :=\Big\{ \{v_1,v_2\} \dots, \{v_{k-1},v_k\}\Big\}$ and let $v_i$ (resp. $v_j$) be the node in $V(K) \in T(\ell)$
with the smallest (resp. biggest) index. 

\item[2.2]  Do the following move:

\vspace{.2cm}
\hspace*{.3cm}
($\oplus$) Change the coordinates of $\ell$  by setting
\begin{equation*}
\ell_e =
\begin{cases}
1 & \text{if } e \in E(H)\setminus \mathcal M_{\ell} \\
0 & \text{if } e \in E(H)\cap \mathcal M_{\ell}  \\
\frac{1}{2} & \text{if } v_i \mbox{ is a single witness for some $C \in \mathcal C_{\ell}\setminus C_y$ and } e \in E(C)\\
\frac{1}{2} & \text{if } v_j \mbox{ is a single witness for some $\bar C \in \mathcal C_{\ell}\setminus C_y$ and } e \in E(\bar C)\\
\end{cases}
\end{equation*}
\hspace{.3cm} and let the nodes $v_i$ and $v_j$ use their tokens to pay for this move;

\end{enumerate}

\vspace*{.1cm}
\noindent
3. While there exists a non-dangerous cycle-component $K$ of $\mathcal T^{\ell}$:
\begin{enumerate}
\item[3.1] Let $v$ and $u$ be two nodes in $V(K) \in T(\ell)$, 
such that $v,u$ are not witnesses in $\mathcal W$, and do the following move:

\vspace{.2cm}
\hspace*{.3cm}
($\ominus$) Change the coordinate $\ell_e$ of all the edges of $E(K)$ by setting
\begin{equation*}
\ell_e =
\begin{cases}
1 & \text{if } e \in E(H)\setminus \mathcal M_{\ell} \\
0 & \text{if } e \in E(H)\cap \mathcal M_{\ell}  \\
\end{cases}
\end{equation*}
\hspace{.3cm} and let the node $v$ and $u$ use their tokens to pay for this move;

\end{enumerate}

\vspace*{.1cm}
\noindent
4. While there exists a dangerous cycle-component $K$ of $\mathcal T^{\ell}$:
\begin{enumerate}
\item[4.1] Let $P$ be a $K$-linking path between two distinct cycles $C$ and $\bar C$ in $\mathcal C_y \setminus \mathcal C_{\ell}$.
Let $(u,v)$ be the pair of witnesses for $C$ and let $(\bar u,\bar v)$ be the pair of witnesses for $\bar C$.

\item[4.2] Do the following moves:

\vspace{.2cm}
\hspace*{.3cm}
($\otimes$) Change the coordinate $\ell_e$ of all the edges of $P$ by setting
\begin{equation*}
\ell_e =
\begin{cases}
1 & \text{if } e \in E(P)\setminus \mathcal M_{\ell} \\
0 & \text{if } e \in E(P)\cap \mathcal M_{\ell}  \\
\end{cases}
\end{equation*}
\hspace{.3cm} and let the nodes $u$ and $v$ use their tokens to pay for this move;

\vspace{.2cm}
\hspace*{.3cm}
($\oslash$) Change the coordinate $\ell_e$ of all the edges in $E(C) \cup E(\bar C) \cup (E(K)\setminus \{v, \bar v\})$ by setting
\begin{equation*}
\ell_e =
\begin{cases}
1 & \text{if } e \in (E(K)\setminus (E(P \cup C \cup \bar C))) \setminus \mathcal M_{\ell} \\
0 & \text{if } e \in (E(K)\setminus (E(C \cup \bar C)))\cap \mathcal M_{\ell}  \\
\frac{1}{2} & \text{if } e \in E(C)\cup E(\bar C)  \\
\end{cases}
\end{equation*}
\hspace{.3cm} and let the nodes $\bar u$ and $\bar v$ use their tokens to pay for this move;

\end{enumerate}

\vspace*{.1cm}
\noindent
5. While there exists a cycle $C \in \mathcal C_{y}\setminus C_{\ell}$, do the following move:

\vspace{.2cm}
\hspace*{.8cm}
($\odot$) Change the coordinate $\ell_e$ of all the edges in $E(C) $ by setting
\begin{equation*}
\ell_e = \frac{1}{2} \qquad \text{ if } e \in E(C)  
\end{equation*}
\hspace{.3cm} and let the witnesses of $C$ use their tokens to pay for this move;

\vspace*{.1cm}
\noindent
6. Output $\ell$. 

\vspace*{-.2cm}
\noindent
\hrulefill

\bigskip
\begin{lemma}
All steps of Algorithm 3 can be correctly performed, invariant (\ref{eq:invariant_6bis}) and (\ref{eq:invariant_final}) are maintained, and the vertex output in Step 6 satisfies $\ell =y$.
\end{lemma}

\begin{proof}
First, we argue that each step of Algorithm 3 can be correctly performed and that the claimed invariants 
are maintained, by induction on the number of moves performed by the algorithm.

In Step 2, we consider a path-component $K$ of $\mathcal T^{\ell}$. 
Invariant (\ref{eq:invariant_final}), which holds by the inductive hypothesis, guarantees that $K$ is a component of $\mathcal T^{r}$, and that
if any of the endpoints is a single witness of a cycle, then this cycle is packed by $\ell$. 
Therefore, the move ($\oplus$) is a valid move, according to either Lemma~\ref{cor:adjacency}(b), or Lemma~\ref{cor:adjacency}(e),
 or Lemma~\ref{cor:adjacency}(f).
Furthermore, since $V(K) \cap T(\ell) = V(K) \cap T(r)$, property (a) of Lemma~\ref{lem:r_properties}
guarantees that there are at least two nodes in $V(K)$ with an available token which can pay for the move.

We now show that the invariants are maintained.
Let $v_i$ be the node chosen in Step 2.1.
We claim that if $v_i$ is a witness, then it can only be a single witness.
Assume for a contradiction that
$v_i$ is in a pair of witnesses $(v_i, v_{i'}) \in \mathcal W$,
with $i \neq i'$ for some cycle $C \in \mathcal C_y\setminus \mathcal C_{\ell}$.
Note that, $Q(v_i, v_{i'}) \subseteq K$  
by property (c) of Lemma~\ref{lem:r_properties}. 
Since $v_{i'} \in T(\ell)$ by invariant (\ref{eq:invariant_6bis}), it follows $i'> i$. 
Then, by definition of target matching, $\{v_i, v_{i+1}\} \in M_C$, and since
$v_i $ has degree 2 in $K$ by property (c) of Lemma~\ref{lem:r_properties}, 
it follows that $\{v_{i-1}, v_{i}\} \in \mathcal M_{\ell}$. However, 
this also implies that $v_{i-1} \in T(\ell)$ (using invariant (\ref{eq:invariant_final})
 together with (\ref{eq:invariant_4bis}) which holds for $r$). We get a contradiction
 with our choice of $i$. We can apply a similar argument to $v_j$. This yields
that invariant (\ref{eq:invariant_6bis}) is maintained.
Finally, after the move is performed, the component $K$ is no longer a component of the target graph, while
all other components of the target graph remain the same. This shows that  invariant (\ref{eq:invariant_final})
is maintained.

In Step 3, we consider a non dangerous cycle-component $K$ of $\mathcal T^{\ell}$. 
The move ($\ominus$) is a valid move, according to Lemma~\ref{cor:adjacency}(a).
Since $K$ is not dangerous, there is at least one node
in $V(K) \cap T(\ell) = V(K) \cap T(r)$ that is not the witness of any cycle.
However, using invariant (\ref{eq:invariant_4bis}), which holds for $r$,
we know that $|V(K) \cap \mathcal T(\ell)|$ is even.
Furthermore, since $K$ is a cycle, $K$ also contains an even number of nodes that
are witnesses in $\mathcal W$. It follows that $K$
contains at least two nodes
in $V(K) \cap T(\ell)$ that are not witnesses of any cycle,
and therefore can pay for the move.
This shows that invariant (\ref{eq:invariant_6bis}) is maintained.
After the move is performed, the component $K$ is no longer a component of the target graph, while
all other components of the target graph remain the same. This shows that  invariant (\ref{eq:invariant_final})
is maintained.

In Step 4, we consider a dangerous cycle-component $K$ of $\mathcal T^{\ell}$. 
Let us argue that Step 4.1 can be performed, i.e. that there exists a $K$-linking path $P$. 
Let $e = \{v,\bar v\}$ be any edge in $\mathcal M_{\ell}$ with both endpoints in
$T(\ell)$. Such an edge exists, because  
invariant (\ref{eq:invariant_final}), which holds by the inductive hypothesis, guarantees that $K$ is a component of $\mathcal T^{r}$
and $V(K) \cap T(\ell) = V(K) \cap T(r)$, therefore we can rely on property (a) of Lemma~\ref{lem:r_properties},
and on invariant (\ref{eq:invariant_4bis}) which holds for $r$.
Since the component is dangerous, $v$ and $\bar v$ are both witness nodes. Let $C \in \mathcal C_y\setminus \mathcal C_{\ell}$ 
be the cycle which has $v$ as one of its witnesses. 
Note that $\bar v$ cannot be a witness for the same cycle $C$: if this was the case, then 
either (i) $\{v,\bar v\} = Q(v,\bar v)$, implying that  $\{v,\bar v\} \subseteq M_{C}$,
contradicting that the edge is in the target graph, or (ii) $K= Q(v,\bar v) \cup \{v,\bar v\}$, contradicting that
$K$ contains the witnesses of at least two distinct cycles.
 Let $\bar C$ be the cycle which has $\bar v$ as one of its witnesses.
Neither $v$ or $\bar v$ can be a single witness, since such nodes appear only in path-components, by property (b) of Lemma~\ref{lem:r_properties}.
Let $u$ and $\bar u$ be such that $(u, v) \in \mathcal W$, and $(\bar u, \bar v) \in \mathcal W$.
By invariant (\ref{eq:invariant_final}), we know all the components of 
$\mathcal T^{\ell}$ are dangerous cycle-components. By property (d) of Lemma~\ref{lem:r_properties}, 
$V(C) \cap V(K') =\emptyset$ for all other components $K' \neq K$ in $\mathcal T^{\ell}$. It follows that  
$(M_C \Delta \mathcal M_{\ell}) \cap E(C) = Q(u, v)$, where $M_C$ is the target matching of $C$.
This implies that there exists a $\mathcal M_{\ell}$-alternating path $P_1$ from $v$ to the unique $M_C$-exposed node in $V(C)$,
which does not use edges in $Q(u,v)$ (refer again to Figure~\ref{fig:dang}). The same argument shows that  there exists a $\mathcal M_{\ell}$-alternating path $P_2$ from $\bar v$ to the unique $M_{\bar C}$-exposed node in $V(\bar C)$,
which does not use edges in $Q(\bar u, \bar v)$. Combining $P_1$, $\{v \bar v\}$, and $P_2$, yields a $K$-linking path $P$.
The move ($\otimes$) is then a valid move, according to Lemma~\ref{cor:adjacency}(b), 
and $u$ and $v$ have their token available by invariant (\ref{eq:invariant_6bis}). 

After the move ($\otimes$) is performed, $C$ becomes packed, with $u$ being its unique ($\mathcal M_{\ell} \cap E(C)$)-exposed node,
and $\bar C$ becomes packed, with $\bar u$ being its unique ($\mathcal M_{\ell} \cap E(\bar C)$)-exposed node.
The move ($\oslash$) is a then valid move, according to Lemma~\ref{cor:adjacency}(f).
The nodes $\bar u$ and $\bar v$ have their token available by invariant (\ref{eq:invariant_6bis}). 
After the move ($\oslash$) is performed, the component $K$ is no longer a component of the target graph, while
all other components of the target graph remain the same. This shows that  invariant (\ref{eq:invariant_final})
is maintained. Since both $C$ and $\bar C$ are now in the support of the current vertex, invariant (\ref{eq:invariant_6bis}) is maintained.

Let $\ell$ be the vertex after Step 4 terminates. Note that all cycles in $\mathcal C_y \setminus \mathcal C_{\ell}$ are packed
by $\mathcal M_{\ell}$, and $\mathcal M_{\ell} \cap \mathcal M_y = \mathcal M_y$. 
All cycles in $\mathcal C_y\setminus \mathcal C_{w}$ which had a unique witness in $\mathcal W$, have appeared 
in the support of the current vertex during 
some iteration of Step 2, because  of
invariant (\ref{eq:invariant_final}) together 
with property (b) of Lemma~\ref{lem:r_properties}.
Therefore, all cycles in $\mathcal C_y \setminus \mathcal C_{\ell}$ 
have two distinct witnesses, which have their token
available at the beginning of Step 5 because of invariant (\ref{eq:invariant_6bis}). 
It follows that every move ($\odot$), which is a valid move by Lemma~\ref{cor:adjacency}(c), can be paid.
It is easy to see then that the output vertex is $y$.
\end{proof}

\subsection{Lower bound}
\label{sec:lower_bound}
We here argue that the quantity $\max_{x \in vert(\mathcal P_{FM})} \{ {\bf 1}^Tx + \frac{|\mathcal C_x|}{2}\}$ is a lower bound on the value of the diameter of $\mathcal P_{FM}$. 

Let $w$ be the vertex at which the above maximum value is achieved. We will show that the distance between the vertex $w$ and the ${\bf 0}$-vertex (i.e. the vertex corresponding to an empty matching), is at least  ${\bf 1}^Tw + \frac{|\mathcal C_{w}|}{2}$.

Suppose to introduce a non-negative slack variable for  each inequality of the form
$x(\delta(v)) \leq 1$ of $\mathcal P_{FM}$.
We get a polytope that naturally corresponds to the set of feasible solutions of a fractional \emph{perfect} matching problem
on a modified graph $\bar G$, defined as follows. Let $\bar G=(V,\bar E)$ be the graph obtained from $G$ by adding a loop edge on each node $v \in V$. We let $\bar E := E \cup L$, with $L$ being the set of loop edges introduced. We can interpret the slack variable associated to a node $v$ as the variable associated to its loop edge $e_v$, and define $\delta_{\bar G}(v) := \delta(v) \cup e_v$ (note that the loop edge is counted once
in this set). We get

$$
\bar{\mathcal{P}}_{FM} :=\{\bar x \in \mathbb R^{\bar E }: \sum_{e \in \delta_{\bar G}(v)} \bar x_e =1, \; \forall v \in V, \; \bar x\geq 0\}
$$

There is a one-to-one correspondence between vertices of $\bar{\mathcal{P}}_{FM}$ and ${\mathcal{P}}_{FM}$.
For a vertex $x$ of $\mathcal{P}_{FM}$ we let $\bar x$ denote the correspondent vertex of $\bar{\mathcal{P}}_{FM}$, and vice versa.
Note that two vertices $x$ and $y$ are adjacent vertices of $\mathcal{P}_{FM}$ if and only if $\bar x$ and $\bar y$ are adjacent vertices of 
$\bar{\mathcal{P}}_{FM}$.

Let $\bar x$ be the vertex of $\bar{\mathcal{P}}_{FM}$ corresponding to $x = {\bf 0}$, i.e. the empty matching in $G$, and
$\bar w$ be the vertex of $\bar{\mathcal{P}}_{FM}$ corresponding to $w$.
The support graph of $\bar x$, denoted by $\bar{\mathcal{G}}_{\bar x}$ contains $|V|$ odd cycles, all of unit length, given by the $|V|$
loop edges. 
The support graph of $\bar w$, denoted by $\bar{\mathcal{G}}_{\bar w}$ contains $|\mathcal C_w| + |V\setminus V(\mathcal G_w)|$
odd cycles: there are $|\mathcal C_w|$ odd cycles of length at least 3, and $|V\setminus V(\mathcal G_w)|$ odd cycles
of unit length, given by loop edges associated to nodes that are not in the support graph $\mathcal G_w$ of $w$.

The following claim will be used to give a lower bound on the distance of two vertices of $\bar{\mathcal{P}}_{FM}$, which depends on
the number of odd cycles that are not in common in the support graphs of the vertices. 

\medskip
\noindent
\emph{Claim.} 
Let $\bar y$ and $\bar z$ be two adjacent vertices of $\bar{\mathcal{P}}_{FM}$. Let $\bar{\mathcal C}_{\bar y}$ be the set of odd cycles in 
the support graph of $\bar y$, and $\bar{\mathcal C}_{\bar z}$ be the set of odd cycles in 
the support graph of $\bar z$. Then $|\bar{\mathcal C}_{\bar y} \Delta \bar{\mathcal C}_{\bar z}| \leq 2$.

\medskip
\noindent
\emph{Proof of claim.}
A proof of the above lemma can be derived from the results given in \cite{B13}. We report the details for completeness.

First we claim that there is at most one component of 
$\bar{\mathcal G}_{\bar y} \cup \bar{\mathcal G}_{\bar z}$
that contains an edge $f$ with $\bar y_f \neq \bar z_f$.
Suppose for a contradiction that there exist two such components, namely $K_1$ and $K_2$.
Let $\tilde z$ be the point defined as $\tilde z_e = \bar z_e$ for all $e \in \bar E\setminus E(K_1)$,
and $\tilde z_e = \bar y_e$ for all $e \in E(K_1)$. Similarly,
let  $\tilde y$ be the point defined as $\tilde y_e = \bar z_e$ for all $e \in \bar E\setminus E(K_2)$,
and $\tilde y_e = \bar y_e$ for all $e \in E(K_2)$. It is not difficult to see that
$\tilde z$ and $\tilde y$ are both vertices of $\bar{\mathcal{P}}_{FM}$.
However, $\frac{1}{2} \bar z + \frac{1}{2} \bar y = \frac{1}{2} \tilde z + \frac{1}{2} \tilde y$.
This is a contradiction, since if two vertices of a polytope are adjacent, there is
a unique way to express their midpoint as a convex combination of vertices.
  
Let $K$ be the component of $\mathcal G_{\bar y} \cup \mathcal G_{\bar z}$
which contains an edge $f$ with $\bar y_f \neq \bar z_f$, and
let $k$ be the number of nodes of this component. If $K$ has at most
$k+1$ edges, then $K$ can be seen to be a tree spanning its $k$ nodes,
plus two additional (possibly loop) edges: it is easy to realize then that $K$ can have at most
2 odd cycles. We are left to show that $K$ contains at most $k+1$ edges.

Let $\bar G[V(K)]$ be the subgraph of $\bar G$ induced by the nodes in $V(K)$.
Consider the fractional perfect matching polytope
associated to the graph $\bar G[V(K)]$.
Now one can see that $\bar z|_{E(\bar G[V(K)])} \in \mathbb R^{E(\bar G[V(K)])}$, obtained from
the vector $\bar z$ by taking only the coordinates in $E(\bar G[V(K)])$,
is a vertex of this polytope, and the same holds for 
$\bar y|_{E(\bar G[V(K)])}$ (defined similarly). Furthermore, these vertices must be adjacent if
$\bar z$ and $\bar y$ are adjacent.

Let $A$ be the incidence matrix of the graph $\bar G[V(K)]$, where we have a row for every
node of $V(K)$, and a column for every edge of $\bar G[V(K)]$. Note that the column associated
to a loop edge has only one non-zero entry. 
Clearly, the rank of $A$ is $|V(K)| =k$.
Furthermore, the constraint matrix of the polytope $\bar{\mathcal{P}}_{FM}$ for the graph
$\bar G[V(K)]$
is $\begin{pmatrix}A\\ I\end{pmatrix} $, with $I$ being the identity matrix of order $E(\bar G[V(K)])$.
Let $M$ be the submatrix of $\begin{pmatrix}A\\ I\end{pmatrix} $ corresponding to the constraints that 
are tight for both $\bar y|_{E(\bar G[V(K)])}$ and $\bar z|_{E(\bar G[V(K)])}$. 
Since $\bar y|_{E(\bar G[V(K)])}$ and $\bar z|_{E(\bar G[V(K)])}$ are adjacent vertices, the rank of $M$
is equal to $|E(\bar G[V(K)])|-1$. Note that $M$ contains $A$ as a submatrix, and
since the rank of $A$ is $k$, it follows that
at least $|E(\bar G[V(K)])|-1 -k$ variables are zero in both $\bar z|_{E(\bar G[V(K)])}$ and $\bar y|_{E(\bar G[V(K)])}$, 
and therefore the union of their support graphs
contains at most $k + 1$ edges.  
\hfill $\square$ 

\smallskip
Now let us discuss how the claim implies our desired lower bound.
By the above claim, if $\bar y$ and $\bar z$ are two (non necessarily adjacent) vertices of $\bar{\mathcal{P}}_{FM}$, then
the quantity $\frac{|\bar{\mathcal C}_{\bar y} \Delta \bar{\mathcal C}_{\bar z}|}{2}$ is a lower bound
on the distance between $\bar y$ and $\bar z$ on the 1-skeleton of $\bar{\mathcal{P}}_{FM}$, since the size of the symmetric difference
of the sets of odd cycles can be reduced by at most 2 at each move. 

We can use this to bound the number of moves needed on a path from $x$ to $w$. 
The distance between $x$ and $w$ on the 1-skeleton of  ${\mathcal{P}}_{FM}$ is equal 
to the distance between $\bar x$ and $\bar w$ on the 1-skeleton of  $\bar{\mathcal{P}}_{FM}$.
The distance between $\bar x$ and $\bar w$ is at least the cardinality of the symmetric difference
of the odd cycles in their support graphs divided by 2, i.e., $(|\mathcal C_w| + |V(\mathcal G_{w})|)/2$.
Note that $\frac{|V(\mathcal G_{w})|}{2} =\sum_{v \in V(\mathcal G_{w})}\frac{1}{2} = {\bf 1}^T w$.
It follows that the distance between $x$ and $w$ is at least ${\bf 1}^T w + \frac{|\mathcal C_w|}{2}$, as desired.

\section{Hardness of computing the diameter of $\mathcal P_{FM}$}
With Theorem \ref{thm:main_result} at hands,  one can easily
prove that computing the diameter of a polytope is strongly NP-hard.

\begin{proof} \emph{(Proof of Theorem \ref{thm:second_result})}
We give an easy reduction from the \emph{Partition Into Triangles} (PIT) problem. In an instance the PIT problem, we are given a simple graph $G=(V,E)$ with $|V| = 3q$ for some integer $q>0$. We want to decide whether there exists a partition of $V$ into $q$ sets $V_1, \dots, V_q$, each containing 3 nodes, with the property that the subgraph of $G$ induced by every set $V_i$ is a triangle. This problem is strongly NP-hard \cite{GJ79}. 

Let $G=(V,E)$ be an instance of PIT, and $\mathcal P_{FM}$ the fractional matching polytope associated to $G$. We claim that the $diam(\mathcal P_{FM}) = \frac{2}{3}|V|$ if and only if there is a yes-answer to the PIT instance described by $G$, i.e. the edges of $G$ allow for a partition of $V$ into triangles.

Suppose that the diameter of $\mathcal P_{FM}$ is $\frac{2}{3}|V|$. From Theorem \ref{thm:main_result}, it follows that there exists a vertex $x$ of $\mathcal P_{FM}$
such that $diam(\mathcal P_{FM}) = {\bf 1}^T x + \frac{|\mathcal C_x|}{2}$. Note that for any $y \in vert(\mathcal P_{FM})$, we have ${\bf 1}^T y \leq \frac{|V|}{2}$ and $|\mathcal C_y| \leq \frac{|V|}{3}$. Therefore, ${\bf 1}^T x + \frac{|\mathcal C_x|}{2} = \frac{2}{3}|V|$ implies that ${\bf 1}^T x = |V|/2$ and $|\mathcal C_x| = |V|/3$. Since $\mathcal C_x$ is a set of node-disjoint odd cycles of a simple graph, it follows that $\mathcal C_x$ contains a set of $|V|/3 =q$ triangles.

Suppose $G$ is a yes-instance, and let $T \subseteq E$ be the set of edges of the triangles induced by the sets $V_i$, for $i=1, \dots, q$.
We construct a fractional matching $x$ by letting $x_e = \frac{1}{2}$ for all $e \in T$, and $x_e =0$ otherwise. 
By construction, ${\bf 1}^Tx = \frac{|V|}{2}$ and $|\mathcal C_x| =\frac{|V|}{3}$, and $x$ is indeed a vertex of $\mathcal P_{FM}$ according to Theorem \ref{thm:balinski}.
As we mentioned before, for any $y \in vert(\mathcal P_{FM})$, we have ${\bf 1}^T y \leq \frac{|V|}{2}$ and $|\mathcal C_y| \leq \frac{|V|}{3}$. Therefore, by Theorem \ref{thm:main_result}, it follows that $diam(\mathcal P_{FM}) = {\bf 1}^T x + \frac{|\mathcal C_x|}{2}=  \frac{2}{3}|V|$.
\end{proof}

 With some extra effort, we can strengthen the above result to show APX-hardness. 
 
 \begin{proof}(\emph{Proof of Theorem \ref{thm:third_result}})
We will show an L-reduction from the optimization version of the problem PIT in graphs with \emph{bounded} degree, i.e. in graphs where the maximum degree $\delta$ of a node
is constant. The optimization version of the problem asks for a set of node-disjoint induced triangles of a graph $G$
of maximum cardinality. This problem is APX-hard \cite{K91}.

Suppose to be given an instance of PIT in a graph $G=(V,E)$ with maximum degree $\delta$. Without loss of generality,
we will assume that each node $v \in G$ is contained in at least one triangle, as otherwise we can just remove that node and its incident edges from the graph. We construct a graph $G'=(V',E')$ as follows (refer to Figure~\ref{fig:G_Gprime}).
We start by setting $V_1 :=\emptyset$ and $E_1 :=\emptyset$.
For each $v \in V$, we add a corresponding node $v$ in $V_1$.
For each induced triangle $t$ of $G$ with nodes $u,v,z$, we add to $V_1$ the set of nodes 
$\{t_1, t_2, t_3, t_4, t_5, t_6, t_7, t_8, t_9\}$  and to $E_1$ the set of edges 
$\Big\{ \{u, t_1\}$, $\{u, t_2\}$, $\{t_1,t_2\}$, $ \{v, t_4\}$, $\{v, t_5\}$, $\{t_4,t_5\}$, $\{z, t_7\}$, $\{z, t_8\}$, $\{t_7,t_8\}$, $
\{t_3, t_1\}$, $\{t_3, t_2\}$, $\{t_4, t_6\}$, $\{t_5, t_6\}$, $\{t_7, t_9\}$, $\{t_8, t_9\}$, $\{t_3, t_6\}$, $\{t_6, t_9\}$, $\{t_3, t_9\} \Big\}$. 
We then take a copy of the graph $(V_1,E_1)$ constructed so far, and denote by $\bar u \in V_2$ the copy of a node $u \in V_1$.
Let the nodes and edges of this copy be $V_2$ and $E_2$, respectively.
We let $V':= V_1 \cup V_2$, and $E':= E_1 \cup E_2 \cup \Big\{ \{v,\bar v\} \mbox{ for all } v \in V \Big\}$.
Let $\mathcal P_{FM}$ be the fractional matching polytope associated to $G'$.
Let $opt_G$ be the value of an optimal solution of the PIT instance, and $opt_{G'}$ be the value of the diameter of
$\mathcal P_{FM}$.

\begin{figure}
\begin{center} 
\includegraphics[angle=0, width=12cm, height=8cm]{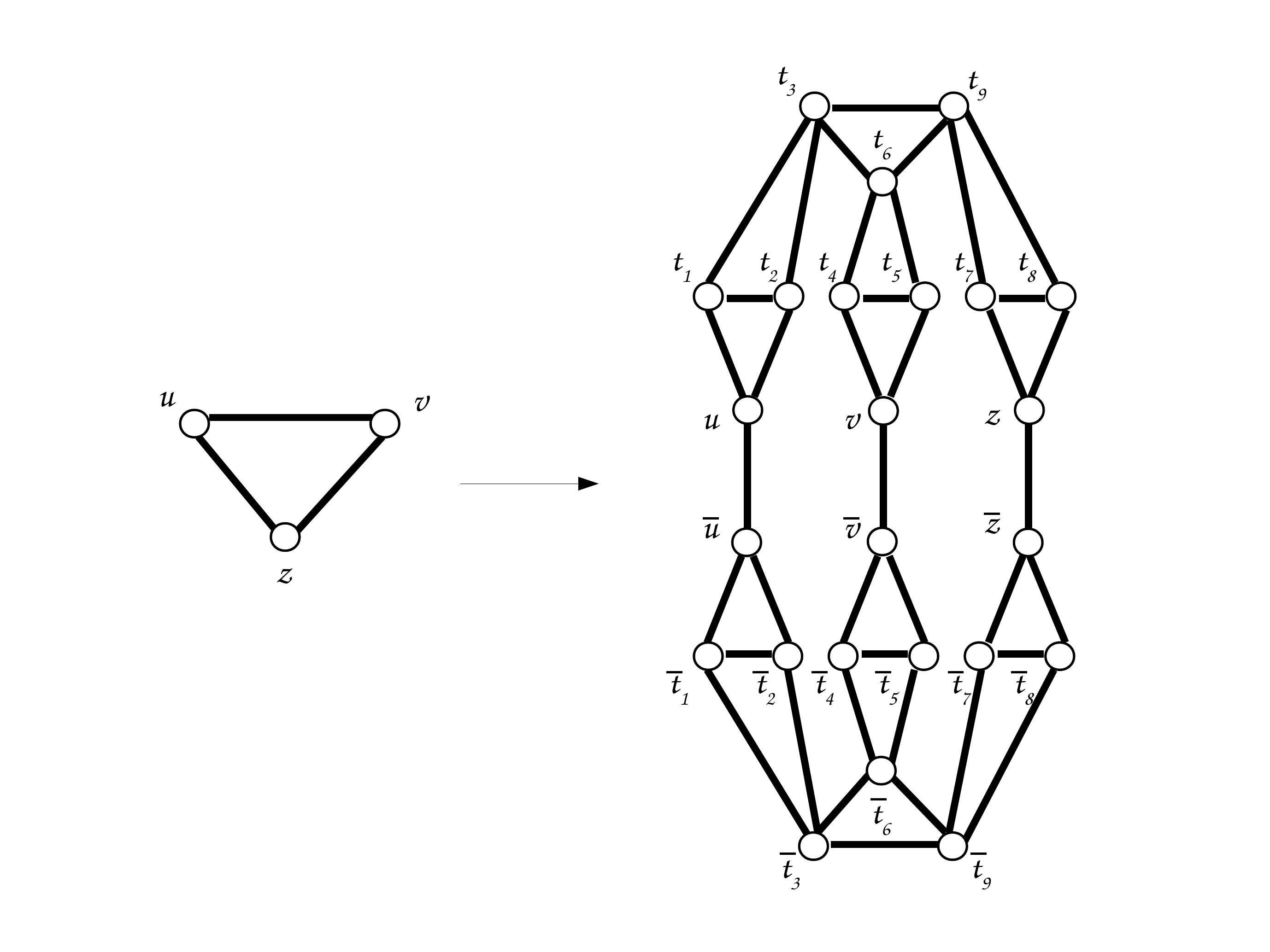}
\end{center}
\caption{Constructing $G'$ from $G$.}
\label{fig:G_Gprime} 
\end{figure}

\bigskip
\noindent
\emph{Claim.} We have $opt_{G'} \leq a \cdot opt_G$, for some constant $a=O(1)$.

\medskip
\noindent
\emph{Proof of Claim.}
Recall that $opt_{G'} = {\bf 1}^T x + \frac{|\mathcal C_x|}{2}$ for some vertex $x$ of $\mathcal P_{FM}$, and
that ${\bf 1}^Tx \leq \frac{|V'|}{2}$, while $\frac{|\mathcal C_x|}{2} \leq \frac{|V'|}{6}$, implying 
$opt_{G'} =O(|V'|)$. Let $k_{tot}$ be the total number of triangles in $G$. Then
$|V'| = 2|V| + 18 \cdot k_{tot}$, and $k_{tot} = O(|V|)$ since each node $v \in V$ can be in at most $\delta^2 =O(1)$ distinct triangles.
It follows that $opt_{G'} =O(|V|)$. We will now argue that $opt_G =\Omega(|V|)$, implying the claim.

Consider a greedy solution $T$ to the PIT instance constructed by doing a sequence of iterations, as follows. 
Start with $T:= \emptyset$.
We call a triangle $t$ of $G$ \emph{bad} for $T$, if $T$ contains
at least one triangle that is not node-disjoint with $t$. 
Similarly, we call a node $v$ \emph{bad} for $T$ if all triangles containing $v$ in $G$ are bad for $T$.
Initialize $B=\emptyset$ to be the set of bad nodes for $T=\emptyset$. Note that 
$B=\emptyset$ because we are assuming that each node in $G$ is contained in at least one triangle.
In each iteration, add an arbitrary triangle $t$ to $T$ that is node-disjoint from all the triangles in the current solution, and
add to $B$ all nodes that become bad in this iteration.
Each time a new triangle $t$ is added to $T$, there are at most $3\delta$ nodes 
that become bad, i.e., the nodes of $t$ plus (possibly) all their neighbors. When the algorithm stops,
the set $B$ contains all nodes in $V$. For this to happen, we need that the number $i$ of iterations
performed by the algorithm is at least $i\geq \frac{|V|}{3\delta}$, implying that our greedy solution
contains $i=\Omega(|V|)$ triangles. Since this is a feasible solution to the PIT instance, we can conclude that 
$opt_G =\Omega(|V|)$ as well. 
\hfill $\diamond$ 

\bigskip
\noindent
\emph{Claim.} Given a pair of vertices at distance $k'$ on the 1-skeleton of $\mathcal P_{FM}$, we can find
in polynomial time $k$ node-disjoint triangles in $G$ such that 
$opt_G -k \leq opt_{G'}-k'$.

\medskip
\noindent
\emph{Proof of Claim.}
Given a pair of vertices at distance $k'$ on the 1-skeleton of $\mathcal P_{FM}$,
by applying Algorithm 1, we can find a vertex $w$ such that $k' \leq {\bf 1}^T w + \frac{|\mathcal C_w|}{2}$. 
We will now compute in polynomial-time a vertex $x$ of $\mathcal P_{FM}$ with the property that
$ {\bf 1}^T x + \frac{|\mathcal C_x|}{2} =  {\bf 1}^T w + \frac{|\mathcal C_w|}{2}$, and in addition, all cycles in 
$\mathcal C_x$ are triangles. We can do this by removing all non-triangle cycles in $\mathcal C_w$, one at the time.

 Let $C \in \mathcal C_w$ be a cycle that is not a triangle (if none exists, $w=x)$. It is not difficult to see 
 that $C$ contains at least one edge of the form $\{t_3,t_6\},\{t_3,t_9\}$ or $\{t_6,t_9\}$
 for some triangle $t$ in $G$.
 Without loss of generality, let us assume that $C$ contains the edge $\{t_3,t_6\}$.

 Suppose that $C$ contains also the edge $\{t_6,t_9\}$.
 Then, let $P$ be the (odd) path obtained by removing from $C$ the nodes $t_3,t_6,t_9$ and their incident edges, 
 let $M$ be a maximum matching of $P$, and
 let $x'$ be the vertex defined as follows:

 \begin{equation*}
 x_e' =
 \begin{cases}
 \frac{1}{2} & \text{if } e \in \big\{ \{t_3,t_6\},\{t_3,t_9\}, \{t_9,t_6\} \big\}  \\
 1 & \text{if } e \in M \\
 0 & \text{if } e \in E(P)\setminus M \\
 w_e & \text{otherwise }  \\
 \end{cases}
 \end{equation*}
 One can see that ${\bf 1}^T x' + \frac{|\mathcal C_{x'}|}{2} = {\bf 1}^T w + \frac{|\mathcal C_w|}{2}$, but $x'$ contains one less
 non-triangle cycle in its support.
 Suppose now that $C$ does not contain $\{t_6,t_9\}$, and therefore it has to contain one edge between $\{t_4,t_6\}$ and $\{t_5,t_6\}$,
 say $\{t_4,t_6\}$. If $\{t_4,t_5\} \in E(C)$, then by letting $P$ be the (odd) path obtained by removing from $C$ the nodes $t_6,t_4,t_5$ and their incident edges, and $M$ be a maximum matching of $P$, we can define a vertex $x'$ as before that again has one less non-triangle cycle in its support.
 Suppose then that $\{t_4,t_5\} \notin E(C)$, and therefore $t_5 \notin V(\mathcal G_w)$. In this case, let $P$ be the (odd) path
 obtained by removing the node $t_4$ from $C$ 
 and its incident edges, 
 let $M$ be a maximum matching of $P$, and
 let $x'$ be the vertex defined as follows:

 \begin{equation*}
 x_e' =
 \begin{cases}
 1 & \text{if } e \in M \cup \{t_4,t_5\} \\
 0 & \text{if } e \in E(P)\setminus M \\
 w_e & \text{otherwise }  \\
 \end{cases}
 \end{equation*}
 One can see that ${\bf 1}^T x' + \frac{|\mathcal C_{x'}|}{2} = {\bf 1}^T w + \frac{|\mathcal C_w|}{2}$, and $x'$ contains one less
 non-triangle cycle in its support.

By repeating the above argument, we can compute a vertex $x$ with the property that all cycles in $\mathcal C_x$ are  triangles, and
\begin{equation}
\label{eq:first}
  {\bf 1}^T x + \frac{|\mathcal C_{x}|}{2} \geq k'
\end{equation}

Necessarily, each $C \in \mathcal C_x$ satisfies either $E(C) \subseteq E_1$
 or $E(C)\subseteq E_2$. Let $\mathcal C^1_x \subseteq \mathcal C_x$ be the set of triangles in $\mathcal C_x$ with edges in $E_1$, and $\mathcal C^2_x \subseteq \mathcal C_x$ be  
 the set of triangles in $\mathcal C_x$ with edges in $E_2$. 
 Note that, for each triangle $t$ in $G$ with nodes $u,v,z$, there can be at most 4 triangles in $\mathcal C^1_x$
whose nodes are in the set $\{t_1, \dots, t_9,u,v,z\}$, and at most 
4 triangles in $\mathcal C^2_x$
whose nodes are in the set $\{\bar t_1, \dots, \bar t_9, \bar u, \bar v, \bar z\}$.
Let $k_1$ be the number of triangles $t$ of $G$ 
 that satisfy the following property: $\mathcal C^1_x$ contains 4 triangles
whose nodes are in the set $\{t_1, \dots, t_9,u,v,z\}$. 
Similarly, let $k_2$ be the number of triangles $t$ of $G$ 
 that satisfy the following property: $\mathcal C^2_x$ contains 4 triangles
whose nodes are in the set $\{\bar t_1, \dots, \bar t_9, \bar u, \bar v, \bar z\}$. 
Without loss of generality, we can assume $k_1 \geq k_2$.
  
Given $x$, we construct  a feasible solution for the PIT instance as follows. 
We add a triangle $t$ of $G$ to our solution
if $\mathcal C_x$ contains exactly 4 triangles whose nodes are in the set $\{t_1, \dots, t_9,u,v,z\}$.
Any two such triangles must necessarily be node-disjoint, and therefore the solution we construct is indeed a feasible solution for our instance. 
Let $k:=k_1$ be the number of triangles of $G$ in this solution. Furthermore, let $k_{tot}$ be the total number of triangles in $G$. Then

\begin{equation}
\label{eq:second}
  |\mathcal C_x| \leq 6 k_{tot} + 2k
\end{equation}
where the inequality follows since  $|\mathcal C_x|= |\mathcal C^1_x| + |\mathcal C^2_x|$ and $k=k_1 \geq k_2$.

The last ingredient that we need is the following. 
Given an optimal solution $T^*$ for the PIT instance of value $opt_G$, we can construct vertex $\tilde x$ of $\mathcal P_{FM}$ as follows.
For each triangle $t$ of $G$ with nodes $u,v,z$ such that $t \in T^*$, we assign $\tilde x_e=\frac{1}{2}$ to all the edges of the 4 node-disjoint triangles whose nodes are in the set $\{t_1, \dots, t_9,u,v,z\}$, and we assign $\tilde x_e=\frac{1}{2}$ to all the edges of the 4 node-disjoint triangles whose nodes are in the set $\{\bar t_1, \dots, \bar t_9, \bar u, \bar v, \bar z\}$.
For each triangle $t$ of $G$ with nodes $u,v,z$ such that $t \notin T^*$,  we assign $\tilde x_e=\frac{1}{2}$ to all the edges of the 3 node-disjoint triangles whose nodes are in the set $\{t_1, \dots, t_9\}$, and we assign $\tilde x_e=\frac{1}{2}$ to all the edges of the 3 node-disjoint triangles whose nodes are in the set $\{\bar t_1, \dots, \bar t_9\}$. For each node $v \in V$ that is not a node of any triangle in $T^*$, we assign 
$\tilde x_e=1$ for $e=\{v, \bar v\}$. We set $\tilde x_e=0$ to all remaining edges. Then
\begin{equation}
\label{eq:third}
  {\bf 1}^T \tilde x + \frac{|\mathcal C_{\tilde x}|}{2} = \frac{|V'|}{2} + 3 k_{tot} + opt_G
\end{equation}

Putting things together, we have that:
 \begin{equation*}
 \begin{array}{cccc}
opt_G -k & = & {\bf 1}^T \tilde x + \frac{|\mathcal C_{\tilde x}|}{2} - \frac{|V'|}{2} - 3 k_{tot} -k \qquad \qquad& (\mbox{by (\ref{eq:third}))}\\
&\leq & {\bf 1}^T \tilde x + \frac{|\mathcal C_{\tilde x}|}{2} - \frac{|V'|}{2} - 3 k_{tot} - 
\frac{|\mathcal C_x|}{2} + 3 k_{tot} & (\mbox{by (\ref{eq:second})})\\
&\leq& {\bf 1}^T \tilde x + \frac{|\mathcal C_{\tilde x}|}{2} - {\bf 1}^T x - 
\frac{|\mathcal C_x|}{2}\qquad \qquad \qquad&\\
&\leq& opt_{G'} - k' \qquad \qquad \qquad \qquad \qquad \qquad& (\mbox{by (\ref{eq:first})}) \\
\end{array}
\end{equation*}
\hfill $\diamond$ 

The above two claims yield the desired L-reduction.
 \end{proof}

\bibliographystyle{plain}
\bibliography{references}

\begin{thebibliography}{99}

\bibitem{Ba70}
M.~L. Balinski, \emph{On maximum matching, minimum covering and their
  connections}, Proceedings of the Princeton symposium on mathematical
  programming (1970)

\bibitem{Ba84}
M.~L. Balinski, \emph{The Hirsch conjecture for dual transportation polyhedra}, 
Mathematics of Operations Research, pages 629-633, (1984)

\bibitem{BR74}
M.~L. Balinski, A. Russakoff, \emph{On the assignment polytope}, SIAM Review,
16(4):516-525, (1974)


\bibitem{B13}
R.~E. Behrend, \emph{Fractional perfect b-matching polytopes.
I: General theory}, Linear Algebra and its Applications,
439(12):3822-3858, (2013)

\bibitem{BFH17}
S. Borgwardt, E. Finhold, R. Hemmecke, \emph{Quadratic diameter bounds for
dual network flow polyhedra}, Mathematical Programming, Volume 159, Issue 1-2, 
pages 237-251, (2016)

\bibitem{BDF17}
S. Borgwardt, J.~A. De Loera, E. Finhold. 
\emph{The diameters of network-flow
polytopes satisfy the hirsch conjecture}, Mathematical Programming, (2017)

\bibitem{BHS06}
G. Brightwell, J. Heuvel, L. Stougie. \emph{A linear bound on the diameter of the
transportation polytope}, Combinatorica, pages 133-139, (2006)





\bibitem{C75}
V. Chv\`atal, \emph{On certain polytopes associated with graphs}, 
Journal of Combinatorial Theory, Series B, 18(2):138-154, (1975)

\bibitem{FT94}
A.~M. Frieze, S.~H. Teng, \emph{On the complexity of computing the diameter of a polytope}, 
Comput. Complexity, 4: 207-219, (1994)

\bibitem{GJ79}
M.~R. Garey, D.~S. Johnson, \emph{Computers and Intractability, A Guide to the Theory of
NP-Completeness}, W.H. Freeman and Company, New York, (1979)

\bibitem{E65}
J. Edmonds, \emph{Maximum matching and a polyhedron with 0,1-vertices}, 
Journal of Research of the National Bureau of Standards: Section B Mathematics and
Mathematical Physics, 69B(1-2):125-130, (1965)

\bibitem{EMM14}
P. Eirinakis, D. Magos, I. Mourtos, \emph{From One Stable Marriage to the Next: How Long Is the Way?}, 
SIAM J. Discrete Math., 28(4), 1971-1979, (2014)





\bibitem{H91}
J. Hurkens, \emph{On the diameter of the edge cover polytope}, 
Journal of Combinatorial Theory, Series B
Volume 51, Issue 2, March 1991, Pages 271-276, (1991)


\bibitem{K91}
V. Kann, \emph{Maximum bounded 3-dimensional matching is MAX SNP-complete}, Information Processing Letters 37, pages 27-35 (1991)


\bibitem{KP03}
V. Kaibel, M.~E. Pfetsch, \emph{Some Algorithmic Problems in Polytope Theory}, Algebra, Geometry and Software Systems, 
pages 23-47 (2003)

\bibitem{KW67}
V. Klee, D.~W. Walkup, \emph{The $d$-step conjecture for polyhedra of dimension
$d>6$}, Acta Math., (133):53-78, (1967)

\bibitem{N89}
D. Naddef, \emph{The Hirsch conjecture is true for $(0,1)$-polytopes}. Math. Program.,
45:109-110, (1989)

\bibitem{MS14}
C. Michini, A. Sassano. \emph{The hirsch conjecture for the fractional stable set
polytope}, Mathematical Programming, 147(1):309-330, (2014)

\bibitem{PR74}
M. Padberg, M.~R. Rao, \emph{The travelling salesman problem and a class of polyhedra
of diameter two}, Mathematical Programming, pages 32-45, (1974)

\bibitem{RC98}
F.~J. Rispoli, S. Cosares, \emph{A bound of 4 for the diameter of the symmetric
traveling salesman polytope}, SIAM Journal on Discrete Mathematics, pages 373-380 (1998)


\bibitem{N17}
N. Sukegawa, \emph{Improving bounds on the diameter of a polyhedron in high dimensions},
Discrete Mathematics, Volume 340, Issue 9, Pages 2134-2142, (2017)

\bibitem{S12}
F. Santos, \emph{A counterexample to the Hirsch conjecture}, Ann. of Math. (2), 176(1):383-412, (2012)

\bibitem{S13}
F. Santos, \emph{Recent progress on the combinatorial diameter of polytopes and simplicial
complexes}, Top 21, pages 426-460, (2013)

\bibitem{S00}
S. Smale, \emph{Mathematical problems for the next century}, Mathematics:
frontiers and perspectives, American Mathematics Society, Providence, RI, pages 271-294 (2000)







\end{thebibliography}

 \section*{Appendix}

\paragraph{Example A.} In Figure \ref{fig:exampleA}, black dashed edges represent the edges in the support of a vertex $y$, and gray dashed edges represent the edges in the support of a vertex $z$, all of value $\frac{1}{2}$. Here the diameter bound is 4, but any path requiring one separate move for each of the 4 triangles, would perform in total at least 5 moves.

\begin{figure}[H]
\begin{center} 
\includegraphics[scale=0.18]{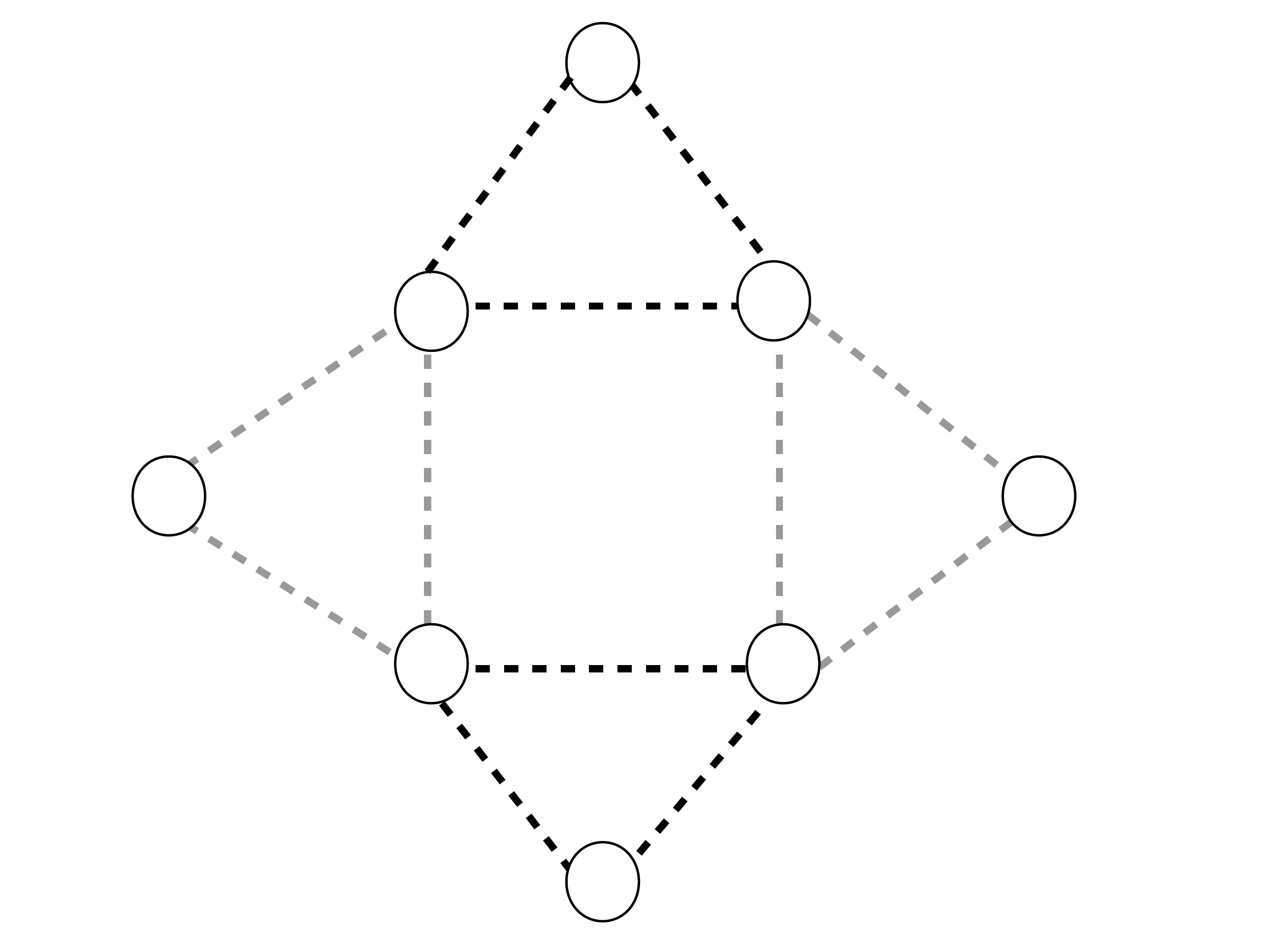}
\end{center}
\caption{Example showing that paths using only moves described in Lemma~\ref{cor:adjacency}(a),(b),(c), might exceed the diameter bound.}
\label{fig:exampleA} 
\end{figure}

\paragraph{Proof of Lemma \ref{cor:adjacency}.}
\begin{proof}
In each of the cases $(a)-(f)$, we will provide a
cost vector $c \in \mathbb R^E$, such that if one considers the optimization problem $\max \{c^Tx : x \in \mathcal P_{FM}\}$, then $z$ and $y$ are the unique optimal basic solutions. This implies that $z$ and $y$ are adjacent vertices.
In all cases, we are going to choose $c$  such that $c_e =1$ for all $e$ such that $ e \notin E(C) \mbox{ and } z_e >0 $, and 
we are going to choose $c$  such that $c_e < 0$ for all edges $e$ such that  $e \notin E(C) \mbox{ and } z_e = 0$.
The values for the coordinates of the edges of the component $C$ will be chosen as follows.
 
For $(a)$, one can choose $c$ such that $c_e =1$ for $e \in E(C)$.
For $(b)$, assume without loss of generality that $|\mathcal M_y \cap E(C)| \leq |\mathcal M_z \cap E(C)|$. One can choose 
$c$ such that $c_e =1$ for $e \in E(C) \cap \mathcal M_{y}$,
 $c_e = \frac{|\mathcal M_y \cap E(C)|}{|\mathcal M_z \cap E(C)|}$ for $e \in E(C) \cap \mathcal M_{z}$.
For $(c)$, assume without loss of generality that  $C$ is one odd cycle in $\mathcal C_z(y)$.
One can choose $c$ such that $c_e =1$ for $e \in E(C) \cap \mathcal M_{y}$,
$c_e = \frac{2|\mathcal M_y \cap E(C)|}{|E(C)|+1}$ for $e \in E(C) \setminus \mathcal M_{y}$.
For $(d)$ and for $(e)$, one can choose $c$ such that $c_e =1$ for $e \in E(C)$.
For $(f)$, assume without loss of generality that the odd cycle in $C$ belongs to the set $\mathcal C_z(y)$. 
One can choose $c$ such that $c_e =1$ for $e \in E( C) : z_e > 0$,
 $c_e = \frac{|\mathcal M_z \cap E(P)| + 0.5}{|\mathcal M_y \cap E(P)|}$ for $e \in E(P) \cap \mathcal M_{y}$.
In all such cases, one can see that $z$ and $y$ are the only optimal solutions that have the structure
described in Theorem \ref{thm:balinski}.
\end{proof}

\end{document}